\documentclass[10pt,
]{article}
\usepackage[
 a4paper, scale=0.78,
]{geometry}
\usepackage{comment} 

\newenvironment{stripalign}{}{}
\newcommand\aptfigure[1]{\includegraphics[height=0.95\textheight]{#1}}
\newcommand\alt[1]{}

\usepackage{amsmath}
\usepackage{amsfonts}
\usepackage{amsthm}
\usepackage{amssymb}
\usepackage{latexsym}
\usepackage{fixmath}
\usepackage{mathdots}

\usepackage{xcolor}
\usepackage{graphicx}

\usepackage{siunitx}
\usepackage{array}
\usepackage{tikz}
\usepackage{rotating}

\usepackage{hyperref}
\hypersetup{colorlinks}  
\usepackage[capitalize]{cleveref}  
\crefname{assumption}{assumption}{assumptions}
\Crefname{figure}{Figure}{Figures}
\crefname{equation}{}{}
\Crefname{equation}{Eq.}{Eqs.}

\usepackage[shortlabels]{enumitem}
\setlist[enumerate,2]{label=(\alph*),ref=\theenumi(\alph*)}
\setlist[enumerate,3]{label=\roman*.,ref=\theenumii\roman*}

\usepackage{algorithm,algorithmic}

\usepackage{mathtools}

\DeclarePairedDelimiterXPP\set[1]{}{\{}{\}}{}{

#1}

\DeclarePairedDelimiter\N{\|}{\|}

\def\ub{\underbrace}
\def\ol{\overline}
\def\ul{\underline}
\def\wtd{\widetilde}
\def\what{\widehat}

\def\ee{\mathrm{e}}
\DeclareMathOperator{\diff}{d\!}

\DeclareMathOperator{\subspan}{span}

\DeclareMathOperator{\T}{T}
\DeclareMathOperator{\OO}{O}
\DeclareMathOperator{\oo}{o}

\newtheorem{theorem}{Theorem}[section]

\theoremstyle{definition}

\newtheorem{example}{Example}[section]

\numberwithin{equation}{section}
\numberwithin{figure}{section}
\numberwithin{table}{section}
\allowdisplaybreaks

\newcommand\clue[2]{\stackrel{\makebox[0pt][c]{\scriptsize #1}}{#2}\;}

\DeclareMathOperator{\nnz}{nnz}
\DeclareMathOperator{\nres}{NRes}
\DeclareMathOperator{\trun}{tr}
\DeclareMathOperator{\opE}{\mathrm{E}}

\usepackage{amsbsy}

\usepackage{mathrsfs}
\newcommand\op[1]{\mathscr{ #1}}

\newcommand\ulub[4]{#1\underline{#2}#3#4}
\newcommand\ubul[4]{#1#2\underline{#3}#4}
\DeclareMathOperator{\mul}{\textsc{m}\textrm{ul}}
\DeclareMathOperator{\sol}{\textsc{s}\textrm{ol}}

\title{An RADI-type method for stochastic continuous-time algebraic Riccati equations
	\thanks{Submitted to the editors \today.}
}

\author{
Zhen-Chen Guo\thanks{School of Mathematical Sciences, Nankai University, Tianjin 200071, China.
\texttt{e-mail: guozhenchen@nankai.edu.cn}. 
Supported in part by NSFC-12371380.
} \and 
Xin Liang\thanks{Corresponding author. Yau Mathematical Sciences Center, Tsinghua University, Beijing 100084, China, and 
Yanqi Lake Beijing Institute of Mathematical Sciences and Applications, Beijing 101408, China;
\texttt{e-mail: liangxinslm@tsinghua.edu.cn}.
Supported in part by NSFC-12371380.
} 
}

\date{\today}

\begin{document}
\maketitle

\begin{abstract}
	In this paper, we propose an RADI-type method for large-scale stochastic continuous-time algebraic Riccati equations with sparse and low-rank matrices.
	This new variant of RADI-type methods is developed by integrating the core concept of the original RADI method with the implicit appearance of the left semi-tensor product in stochastic continuous-time algebraic Riccati equations.
	The method employs different shifts to accelerate  convergence and uses compression techniques to reduce storage requirements and computational complexity.
	Unlike many existing methods for large-scale problems such as Newton-type methods and homotopy method, it calculates the residual at a low cost and does not require a stabilizing initial approximation, which can often be challenging to find.
	Numerical experiments are provided to demonstrate its efficiency.
\end{abstract}



\section{Introduction}\label{sec:introduction}
Consider the stochastic continuous-time algebraic Riccati 
equations (SCAREs)
\begin{equation}\label{eq:scare}
	\begin{multlined}[b]
		Q+A_{{0}}^{\T}X+XA_{{0}} + \sum_{i=1}^{{r-1}} A_i^{\T}XA_i
		\\=(L+XB_{{0}}+\sum_{i=1}^{{r-1}} A_i^{\T}XB_i)
		(R+\sum_{i=1}^{{r-1}} B_i^{\T}XB_i)^{-1}
		\alt{\\\cdot}
		(L^{\T}+B_{{0}}^{\T}X+\sum_{i=1}^{{r-1}} B_i^{\T}XA_i)
		,
	\end{multlined}
\end{equation}
where $A_i\in \mathbb{R}^{n\times n}, B_i\in \mathbb{R}^{n\times m}$ for $i=0,1,\dots, r-1$, $R\in \mathbb{R}^{m\times m}$ is positive definite, 
$L\in \mathbb{R}^{n\times m}$ and $Q\in \mathbb{R}^{n\times n}$ with
$\begin{bmatrix}
	Q& L \\ L^{\T} & R
\end{bmatrix}$ being positive semi-definite.
Here $r-1$ is the number of stochastic processes involved in the stochastic systems.
Clearly, when $r=1$ the SCAREs reduce 
to the classical continuous-time algebraic Riccati equations (CAREs):
$Q+A^{\T}X+XA-(L+XB)R^{-1}(L^{\T}+B^{\T}X)=0.$

The SCARE \cref{eq:scare} arises from the stochastic time-invariant control system in continuous-time subject to multiplicative white noise: 
\begin{equation}\label{eq:scare-system:scare}
	\diff x(t) 
	= (A_{0}x(t) + B_{0}u(t))\diff t+\sum_{i=1}^{r-1}(A_ix(t)+B_iu(t))\diff w_{i}(t),
\end{equation}
in which $x(t),u(t)$ 
are the state and control, respectively,
and $w(t)=\begin{bmatrix}
	w_{1}(t) &\!\cdots\! &w_{r-1}(t)
\end{bmatrix}
	^{\T}$ is a  standard Wiener process 
	where
	each $w_i(t)$ is a standard Brownian motion and the $\sigma$-algebras $\sigma\left(w_i(t), t\in [t_0,\infty)\right), i=1,\dots,r-1$ are independent \cite{draganMS2013mathematical}.
	Considering the cost functional with respect to the control $u(t)$ with the given initial $x_0$:
\begin{equation}\label{eq:cost-function:scare}
	\alt{\!\!}
	J(t_0,x_0;u) \alt{\!}=\alt{\!} 
	\opE{\alt{\!\!}\int_{t_0}^{\infty}
		\alt{\!\!}
		\begin{bmatrix}
			x_{t_0,x_0;u}(t)\alt{\!\!} \\ u(t)
		\end{bmatrix}^{\alt{\!\!}\T}
		\alt{\!\!}
		\begin{bmatrix}
			Q & \alt{\!\!\!}L \\ L^{\alt{\!}\T} & \alt{\!\!\!}R
		\end{bmatrix}
		\alt{\!\!}
		\begin{bmatrix}
			x_{t_0,x_0;u}(t)\alt{\!\!} \\ u(t)
	\end{bmatrix}
		\alt{\!}
\diff t},\alt{\!}
\end{equation}
where $x_{t_0,x_0;u}(t)$ is the solution of the system \cref{eq:scare-system:scare} corresponding to the control $u(t)$ with the initial $x_{t_0,x_0;u}(t_0)=x_0$, one goal in stochastic control is to minimize the cost functional \cref{eq:cost-function:scare} and derive an optimal control. Such an optimization problem is also known as the first linear-quadratic optimization problem 
{\cite[Section~6.2]{draganMS2013mathematical}}.    

We impose the following assumptions throughout this paper:

\begin{enumerate}[(\mbox{C}1)]
	\item \label{item:item1:scare} $R\succ 0$;
	\item \label{item:item2':scare} 
		the pair
		$(\set{A_i}_{i=0}^{r-1}, \set{B_i}_{i=0}^{r-1})$ is stabilizable, i.e.,
		there exists $F\in \mathbb{R}^{m\times n}$ such that the adjoint of the operator $\ee^{\op L_F(t-t_0)}$ is stable,
		namely $\ee^{\op L_F^*(t-t_0)}$ is stable where
		\begin{equation*}
		\begin{aligned}[b]
			\MoveEqLeft[0]
				\op L_F^*\colon S \mapsto 
				(A_{0}+B_{0}F)^{\T}S+S(A_{0}+B_{0}F)
				\alt{\\&\qquad\qquad\qquad}+\sum_{i=1}^{r-1}(A_i+B_iF)^{\T}S(A_i+B_iF)
				;
		\end{aligned}
		\end{equation*}
	\item \label{item:item3:scare} the pair
		$(\set{A_i}_{i=0}^{{r-1}}, C)$ is detectable 
		with $C^{\T}C\alt{\!}=\alt{\!}Q-LR^{-1}L^{\T}$, or equivalently, $(\set{A_i^{\T}}_{i=0}^{r-1}, \set{C_i^{\T}}_{i=0}^{r-1})$ is stabilizable with $C_0=C,C_i=0$ for $i=1,\dots, r-1$. 
\end{enumerate}
It is known that if the assumptions above hold,
then \cref{eq:scare} has a unique 
positive semi-definite stabilizing solution $X_{\star}$, see, e.g., {\cite[Theorem~5.6.15]{draganMS2013mathematical}}.
Here, $X$ is a stabilizing solution if the system $(A_{0}+B_{0}F_X, A_1+B_1F_X, \dots, A_{r-1}+B_{r-1}F_X)$ is stable  with 
\begin{equation}\label{eq:F:scare}
	F_X\alt{\!}=\alt{\!}-(R\alt{\!}+\alt{\!}\sum_{i=1}^{r-1} B_i^{\T}XB_i)^{\alt{\!}-\alt{\!}1}(L^{\T}\alt{\!}+\alt{\!}B_{0}^{\T}X\alt{\!}+\alt{\!}\sum_{i=1}^{r-1} B_i^{\T}XA_i), 
\end{equation}
or equivalently, $\op L_{F_X}$ is exponentially stable with the associated $F_\star=F_{X_\star}$ taking the feedback control specified in \cref{eq:F:scare} with $X=X_\star$. 
In fact,  $X_{\star}$ is  a stabilizing solution if and only if the zero equilibrium of the closed-loop system 
\begin{equation}\label{eq:close-loop:scare}
\diff x(t) = (A_{0} + B_{0}F_{\star})x(t)\diff t+\sum_{i=1}^{r-1}(A_i+B_iF_{\star})x(t)\diff w_{i}(t)
\end{equation}
is strongly exponentially stable in mean square \cite[Remark~5.4.1]{draganMS2013mathematical}, where $F_\star=F_{X_\star}$ is as in \cref{eq:F:scare} with $X=X_\star$.
Furthermore, the cost functional \cref{eq:cost-function:scare} admits an  optimal control $u(t)=F_{\star}x_{t_0,x_0}(t)$ where $x_{t_0,x_0}(t)$ is the solution to the corresponding closed-loop system \cref{eq:close-loop:scare}. 

Due to the complicated forms, one may realize the SCAREs would be much more difficult to analyze and solve,
unlike the situation for the CAREs that people have developed rich theoretical results and numerical methods,
see, e.g., \cite{mehrmann1991automomous,lancasterR1995algebraic,ionescuOW1999generalized,biniIM2012numerical,huangLL2018structurepreserving,bennerBKS2020numerical}.
There exist only a few references, e.g., \cite{damm2004rational
,draganMS2013mathematical}, discuss the stochastic linear systems and the induced SCAREs theoretically.
The existing methods for solving SCAREs include Newton's method \cite{damm2004rational,dammH2001newtons}, modified Newton's method \cite{guo2001iterative,ivanov2007iterations,chuLLW2011modified}, Newton's method with generalized Lyapunov/Stein equations solver \cite{fanWC2016smith,ivanov2007properties,takahashiKSS2009numerical}, comparison theorem based method \cite{freilingH2003properties,freilingH2004class}, LMI's (linear matrix inequality) method \cite{ramiZ2000linear,iidukaY2012computational}, and homotopy method \cite{zhangFCW2015homotopy}.
However, these approaches can be computationally expensive and may require finding a stabilizing initial approximation, which can be challenging.

In this paper, we focus on a special case that $A_i$ are large-scale and sparse, and $B_i$ and $Q-LR^{-1}L^{\T}$ are low-rank, which implies $C$ is also low-rank.
With the help of the algebraic structures identified in previous work \cite{guoL2023stochastic},
motivated by the relation illustrated in \cite{guoL2023intrinsic} between the algebraic structure and the efficient RADI method \cite{bennerBKS2018radi} for CAREs,
we 
propose an RADI-type method to solve large-scale SCAREs efficiently.  
Unlike many existing methods for large-scale problems such as Newton-type methods and homotopy method, it computes the residual at a lower cost and does not necessitate a stabilizing initial approximation that may be difficult to find.

The rest of the paper is organized as follows.
We first introduce some notation. 
Then we discuss the idea of incorporation (also known as defect correction) performing on SCAREs and present a low-rank expression of the residual generated by a chosen approximation in \cref{sec:incorporation},  
which derives a prototype of an RADI-type method.
In \cref{sec:implementation-aspects} we show the details that converts the immature method into a practical algorithm.
Numerical experiments and some related discussions are given in \cref{sec:experiments-and-discussions}.
Some concluding remarks in \cref{sec:concluding-remarks} end the paper.

\subsection{Notation}\label{ssec:notation}
In this paper,
${\mathbb R}^{n\times m}$ is the set of all $n\times m$ real matrices,
${\mathbb R}^n={\mathbb R}^{n\times 1}$, and ${\mathbb R}={\mathbb R}^1$.
$I_n$ (or simply $I$ if its dimension is clear from the context) is the $n\times n$ identity matrix. 
For a matrix $X$, $X^{\T}$ is its transpose and $X^{-\alt{\!}\T}=(X^{-1})^{\T}$.
For a symmetric matrix $X$,
$X\succ 0$ ($X\succeq 0$) indicates its positive (semi-)definiteness, and $X\prec 0$ ($X\preceq 0$) if $-X\succ 0$ ($-X\succeq 0$).

By $A \otimes B$ denote the Kronecker product of the matrices $A$ and $B$.
For $A\in \mathbb{R}^{m\times n}, B\in \mathbb{R}^{p\times q}$, define 
the left semi-tensor product of $A$ and $B$:
\[
	A\ltimes B:=
\begin{dcases*}
	(A\otimes I_{p/n})B & if $n\mid p$,\\ 
	A(B\otimes I_{n/p}) & if $p\mid n$. 
\end{dcases*}
\]
Clearly $A\ltimes B=AB$ if $n=p$.
The left semi-tensor product, first defined in 2001 \cite{cheng2001semitensor}, has many applications in system and control theory, such as  Boolean networks \cite{chengQ2009controllability} and electrical systems \cite{xueM2008new}.
It is easy to verify that this product satisfies many of the same 
arithmetic laws as the classical matrix product.
Three useful identities are given below:
	\begin{subequations}
	\label{eq:easy}
\begin{align}
	U\ltimes (I+V\ltimes U)&=(I+U\ltimes V)\ltimes U,\quad
	\\
	U\ltimes (I+V\ltimes U)^{-1}&=(I+U\ltimes V)^{-1}\ltimes U
	, 
\end{align}
	\end{subequations}
\begin{equation}
	\label{eq:smwf}
	\begin{aligned}[b]
		\MoveEqLeft[0]M^{-1} - (M + U\ltimes D\ltimes V)^{-1} =
		\alt{\\&\qquad} M^{-1}\ltimes  U\ltimes  (D^{-1} + V\ltimes  M^{-1} \ltimes U)^{-1} \ltimes V \ltimes M^{-1}.
	\end{aligned}
\end{equation}
The inverses 
in \cref{eq:easy,eq:smwf} assume the 
invertibility.

With the aid of the left semi-tensor product,
 following \cite{guoL2023stochastic},
the SCARE \cref{eq:scare} can be reformulated as an equivalent form, which we call the ``standard'' form: 
\begin{subequations}\label{eq:scare-standard-form:scare}
\begin{equation}
	\op C(X) = 0,
\end{equation}	
where 
the residual operator $\op C(\cdot)$ is defined as\footnote{
``$*$'' denotes the underlined term in the same line; this notation will be frequently used later.}
\begin{align}
	\alt{\!\!}
	\op C(X) 
	\alt{\!}:= 
	&\,C^{\T}C\alt{\!}+\alt{\!} A^{\T}X\alt{\!}+\alt{\!}X A 
		\alt{\!}+\alt{\!} \what A^{\T}\ltimes X\ltimes \what A
		\alt{\nonumber\\ &}	\ulub{\alt{\!}-\alt{\!}}{(X B\alt{\!}+\alt{\!}\what A^{\T}\alt{\!}\ltimes X \ltimes \what B)}{(I\alt{\!}+\alt{\!}\what B^{\T}\alt{\!}\ltimes X \ltimes \what B)^{\alt{\!}-\alt{\!}1}(}{*)^{\alt{\!}\T}\alt{\!}}
		,
\end{align}
\end{subequations}
in which
\begin{align*}
	&
	A = A_{{0}}-B_{{0}}R^{-1}L^{\T}\in \mathbb{R}^{n\times n}, \qquad  
	\\&
	B=B_{{0}}P^{-1}\in \mathbb{R}^{n\times m},  \qquad
	C\in \mathbb{R}^{l\times n},
	\\
	& 
	\what A=\Pi\alt{\!}\left(\alt{\!}\begin{bmatrix}
			A_1 \\ \vdots \\ A_{r-1}
			\end{bmatrix}\alt{\!}-\alt{\!}\begin{bmatrix}
			B_1 \\ \vdots \\ B_{r-1}
		\end{bmatrix}\alt{\!} R^{-1}L^{\T}\alt{\!}\right)\alt{\!}\in \mathbb{R}^{{(r-1)}n\times n},\;
		\\&
		\what B=\Pi \alt{\!}\begin{bmatrix}
		B_1 \\ \vdots \\ B_{r-1}
	\end{bmatrix}\alt{\!}P^{-1}\in \mathbb{R}^{{(r-1)}n\times m}.\alt{\!}
\end{align*}
Here, $\Pi$ is the permutation satisfying $\Pi^{\T}(X\otimes I_{{r-1}}) \Pi = I_{{r-1}}\otimes X$, and $P$ is any matrix such that $P^{\T}P=R$.  
The feedback control $F_X$ and the closed-loop matrix are reformulated as 
\begin{equation}\label{eq:F:hat:scare}
	\begin{aligned}[t]
		F_X
		&={-R^{-1}L^{\T}+P^{-1}\what F_X},
		\\
		\begin{bmatrix}
			A_{{0}}+B_{{0}}F_X
			\\
			A_1+B_1F_X
			\\
			\vdots 
			\\
			A_{r-1}+B_{r-1}F_X
		\end{bmatrix}
		&=
		\begin{bmatrix}
			A
			\\
			\Pi^{\T}\what A
		\end{bmatrix}
		+
		\begin{bmatrix}
			B
			\\
			\Pi^{\T}\what B
		\end{bmatrix}{\what F_X}
		,
	\end{aligned}
\end{equation}
where 
$\what F_X=-(I+\what B^{\T}\ltimes X \ltimes \what B)^{-1}(XB + \what A^{\T}\ltimes X \ltimes  \what B)^{\T}$ is the feedback control of the standard form \cref{eq:scare-standard-form:scare}.

In the following, we will primarily focus on the standard-form SCARE \cref{eq:scare-standard-form:scare}, except where indicated otherwise regarding \cref{eq:scare}.


\section{Incorporation and RADI-type method}\label{sec:incorporation}
The concept of incorporation is as follows: Once an approximate solution $X$ is obtained, writing the difference from the exact solution $X_\star$ as $\Delta=X_\star-X$, the original equation $\op C(X_\star)=0$ can then be transformed into a new equation $\op C_{X}(\Delta)=0$, which would retain the same form and can be solved more efficiently. 
This process can be repeated to obtain more and more accurate approximate solutions.
In this section, we will derive an efficient RADI-type method for SCAREs based on the incorporation philosophy, as what has been done for CAREs in \cite{bennerBKS2018radi}. 

\subsection{Residual and incorporation}\label{ssec:residual-equations}

We first demonstrate that the new equation $\op C_{X}(\Delta)=0$ is indeed of the same form, specifically, also an SCARE. 
To prepare, we compute the difference of the residuals of symmetric matrices $X+\Delta$ and $X$:
\begin{align*}
	&\op C(X+\Delta)-\op C(X)
	\\& =A^{\T}\Delta + \Delta A + \what A^{\T} \ltimes \Delta \ltimes \what A
	\\&\qquad\;
	\ulub{+}
	{(X B + \what A^{\T}\ltimes X \ltimes \what B) } {(I + \what B^{\T}\ltimes X \ltimes \what B)^{-1}(*}{)^{\T}}
	\\&\qquad\;
	\ulub{-\alt{\!}}{([X\alt{\!}+\alt{\!}\Delta] B \alt{\!}+\alt{\!} \what A^{\T}\alt{\!}\ltimes [X\alt{\!}+\alt{\!}\Delta] \ltimes \what B)}{\alt{\!}(I\alt{\!}+\alt{\!}\what B^{\T}\alt{\!}\ltimes [X\alt{\!}+\alt{\!}\Delta] \ltimes \what B)^{\alt{\!}-\alt{\!}1}\alt{\!}(*}{)\alt{\!}^{\T}\alt{\!}.}\alt{\!\!}
\end{align*}
Factorize $\what B^{\T}\ltimes X \ltimes \what B+ I=:R_X=P_X^{\T}P_X$ and write 
\begin{equation}\label{eq:AB-Xi}
	\begin{gathered}[c]
	B_X=BP^{-1}_X,\what B_X=\what BP^{-1}_X,
	\\
		A_X=A-B_XL^{\T}_X, \what A_X=\what A - \what B_XL^{\T}_X,
	\end{gathered}
\end{equation}
where $L^{\T}_X=B_X^{\T}X + \what B_X^{\T}\ltimes X \ltimes \what A$.
Then by using \cref{eq:AB-Xi} and performing some calculations 
we obtain
	\begin{stripalign}
\begin{align}
		\nonumber
		\MoveEqLeft[0]\op C(X+\Delta)-\op C(X)
		\\\nonumber&
		\ulub{= A^{\T}\alt{\!}\Delta \alt{\!}+\alt{\!} \Delta A \alt{\!}+\alt{\!} \what A^{\T}\alt{\!} \ltimes \Delta \ltimes \what A \alt{\!}+\alt{\!}}
		{(X B \alt{\!}+\alt{\!} \what A^{\T}\alt{\!}\ltimes X \ltimes \what B)}{R_X^{\alt{\!}-\alt{\!}1}(*}{)^{\alt{\!}\T}\alt{\!}}
		\nonumber\\\nonumber&\qquad
		\ulub{\alt{\!}-\alt{\!}}{([X\alt{\!}+\alt{\!}\Delta] B \alt{\!}+\alt{\!} \what A^{\T}\alt{\!}\ltimes [X\alt{\!}+\alt{\!}\Delta] \ltimes \what B)}{(R_X \alt{\!}+\alt{\!} \what B^{\T}\alt{\!}\ltimes \Delta \ltimes \what B)^{\alt{\!}-\alt{\!}1}(*}{)^{\alt{\!}\T}}
		\nonumber\\&=
		A^{\T}\alt{\!}\Delta \alt{\!}+\alt{\!} \Delta A \alt{\!}+\alt{\!} \what A^{\T}\alt{\!} \ltimes \Delta \ltimes \what A \alt{\!}+\alt{\!} L_XL^{\T}_X
		\nonumber\\\nonumber&\qquad
		\ulub{\alt{\!}-\alt{\!}}{(L_X\alt{\!}+\alt{\!}\Delta B_X \alt{\!}+\alt{\!} \what A^{\T}\alt{\!}\ltimes \Delta \ltimes \what B_X)}
		{(I\alt{\!}+\alt{\!} \what B_X^{\T}\ltimes \Delta \ltimes \what B_X)^{\alt{\!}-\alt{\!}1}(*}{)^{\T}}
		\nonumber \\&=
		A^{\T}\alt{\!}\Delta \alt{\!}+\alt{\!} \Delta A \alt{\!}+\alt{\!} \what A^{\T}\alt{\!} \ltimes \Delta \ltimes \what A \alt{\!}+\alt{\!} L_XL^{\T}_X
		\nonumber\\\nonumber&\qquad
		\ulub{\alt{\!}-\alt{\!}}{(L_X[I\alt{\!}+\alt{\!}\what B_X^{\T}\ltimes \Delta \ltimes \what B_X] \alt{\!}+\alt{\!}\Delta B_X \alt{\!}+\alt{\!} \what A_X^{\T}\ltimes \Delta \ltimes \what B_X)}
		{(I\alt{\!}+\alt{\!}\what B_X^{\T}\ltimes \Delta \ltimes \what B_X)^{\alt{\!}-\alt{\!}1} (*}{)^{\T}}
		\nonumber\\&=
		A^{\T}\alt{\!}\Delta \alt{\!}+\alt{\!} \Delta A \alt{\!}+\alt{\!} \what A^{\T}\alt{\!} \ltimes \Delta \ltimes \what A \alt{\!}+\alt{\!} L_XL^{\T}_X
		\nonumber\\\nonumber&\qquad
		\begin{multlined}[t]
			\ulub{\alt{\!}-\alt{\!}}{(\Delta B_X \alt{\!}+\alt{\!} \what A_X^{\T}\ltimes \Delta \ltimes \what B_X) L^{\T}_X}{\alt{\!}-\alt{\!}(*)}{^{\T}
			\alt{\!}-\alt{\!}L_X[I\alt{\!}+\alt{\!}\what B_X^{\T}\ltimes \Delta \ltimes \what B_X] L^{\T}_X}
			\\
			\ulub{\alt{\!}-\alt{\!}}{(\Delta B_X \alt{\!}+\alt{\!} \what A_X^{\T}\ltimes \Delta \ltimes \what B_X)}
			{(I\alt{\!}+\alt{\!}\what B_X^{\T}\ltimes \Delta \ltimes \what B_X)^{\alt{\!}-\alt{\!}1}(*}{)^{\T}}
		\end{multlined}
		\nonumber\\&=
		A_X^{\T}\alt{\!}\Delta \alt{\!}+\alt{\!} \Delta A_X \alt{\!}+\alt{\!} \what A_X^{\T} \ltimes \Delta \ltimes \what A_X
		\nonumber\\&\nonumber\qquad
		\ulub{\alt{\!}-\alt{\!}}{(\Delta B_X \alt{\!}+\alt{\!} \what A_X^{\T}\ltimes \Delta \ltimes \what B_X)}
		{(I\alt{\!}+\alt{\!}\what B_X^{\T}\ltimes \Delta \ltimes \what B_X)^{\alt{\!}-\alt{\!}1}(*}{)^{\T}}{}
		\nonumber\\&= 
		\op C_X(\Delta)-\op C_X(0),
	\label{eq:residual-eq}
\end{align}
	\end{stripalign}
where $\op C_X(0)$ can be arbitrarily chosen and
\[
	\op C_X(\Delta)=
	\begin{aligned}[t]
		&\op C_X(0)+
		A_X^{\T}\Delta + \Delta A_X + \what A_X^{\T} \ltimes \Delta \ltimes \what A_X
		\\&\qquad \alt{\!\!}\ulub{\alt{\!}-\alt{\!}}{(\Delta B_X \alt{\!}+\alt{\!} \what A_X^{\T}\ltimes \Delta \ltimes \what B_X) }
		{(I\alt{\!}+\alt{\!}\what B_X^{\T}\ltimes \Delta \ltimes \what B_X)^{\alt{\!}-\alt{\!}1}(*}{)^{\alt{\!}\T}\alt{\!}.}
	\end{aligned}
\]
Let the symmetric matrices $X_\star$ and $X$ be 
an exact solution and an approximate one to \cref{eq:scare-standard-form:scare} respectively, and
write $\Delta\alt{\!}=\alt{\!}X_\star\alt{\!}-\alt{\!} {X}$.
Clearly, $\op C(X_\star)\alt{\!}=\alt{\!}0$.
Then we can construct another SCARE to which $\Delta$ is a solution, as 
is stated in \cref{thm:incor-scare}.

\begin{theorem}\label{thm:incor-scare}
	Given $X\succeq 0$ and let $A_{X},B_{X},\what A_{X},\what B_{X}$ be as in \cref{eq:AB-Xi},
	and let $\op C_{X}(0)=\op C(X)  \succeq 0$.
	Then
	\begin{enumerate}
		\item \label{thm:incor-scare:solution}
			$\op C_{X}(\Delta)=\op C(X+\Delta)$. Moreover,
			$\Delta=X_\star-X$ is a solution to $\op C_{X}(\Delta)=0$, if and only if $\op C(X_\star)=0$.
		\item \label{thm:incor-scare:stabilizing-2}  $\Delta_\star=X_\star-X$ is a stabilizing solution to $\op C_{X}(\Delta)=0$, if and only if $X_\star$ is one to $\op C(X)=0$. 
		\item \label{thm:incor-scare:stabilizing-3} 
			There exists a unique stabilizing solution $\Delta_\star$ to the SCARE $\op C_{X}(\Delta)=0$.
	\end{enumerate}
\end{theorem}
\begin{proof}
	\Cref{thm:incor-scare:solution} follows directly from   \cref{eq:residual-eq}.


	For \cref{thm:incor-scare:stabilizing-2}, 
		according to the discussion in \cref{sec:introduction}, $X_\star$ is a stabilizing solution if and only if the adjoint $\op L^*_{F_\star}$ of the original associated linear differential equation is exponentially stable for
		$F_\star=-(I+\what B^{\T}\ltimes X_\star \ltimes \what B)^{-1}(X_\star B + \what A^{\T}\ltimes X_\star \ltimes  \what B)^{\T}$ (in the form of \cref{eq:F:hat:scare}
		).
	Let $F_{X}=L^{\T}_{X}+P_{X}F_\star$, and then
		\begin{align*}
			A_{X}\alt{\!}+\alt{\!}B_{X}F_{X}
			&=A\alt{\!}-\alt{\!}B_{X}L^{\T}_{X}\alt{\!}+\alt{\!}B_{X}(L^{\T}_{X}\alt{\!}+\alt{\!}P_{X}F_\star)
			=A\alt{\!}+\alt{\!}BF_\star,
			\\
			\what A_{X}\alt{\!}+\alt{\!}\what B_{X} F_{X}
			&=\what A\alt{\!}-\alt{\!}\what B_{X}L^{\T}_{X}\alt{\!}+\alt{\!}\what B_{X}(L^{\T}_{X}\alt{\!}+\alt{\!}P_{X}F_\star)
			=\what A\alt{\!}+\alt{\!}\what BF_\star.
		\end{align*}
	This implies that
	\begin{align*}
			\op L^*_{F_{X}}\alt{\!}(S)
	&
			\alt{\!}:=\alt{\!}\ubul{(*}{
			)^{\alt{\!}\T}\alt{\!}S\alt{\!}+\alt{\!}S}{(A_{{X}}\alt{\!\!}+\alt{\!\!}B_{{X}}F_{X})}{}
		\ubul{\alt{\!}+\alt{\!}(*}{
		)^{\alt{\!}\T}\alt{\!}\ltimes S \ltimes }{(\what A_{X} \alt{\!\!}+\alt{\!\!} \what B_{X} F_{X})}{}
		\alt{\\&}
		\alt{\!}=\alt{\!}\op L^*_{F_\star}\alt{\!}(S),
	\end{align*}
	indicating that  $\op L^*_{F_{X}}$ is exponentially stable.
	Moreover,
	straightforward calculations show that
	\[
		F_{X}
		=-
		(I+\what B_{X}^{\T}\ltimes \Delta_\star \ltimes \what B_{X})^{-1}
		( B_{X}^{\T}\Delta_\star+\what B_{X}^{\T}\ltimes \Delta_\star \ltimes \what A_{X})
	\]
	is the corresponding feedback control associated with the solution $\Delta_\star$ to the SCARE $\op C_{X}(\Delta)=0$.
	Hence $\Delta_\star$ is stabilizing if and only if $X_\star$ is stabilizing.

	\Cref{thm:incor-scare:stabilizing-3} follows from \cref{thm:incor-scare:stabilizing-2} and the fact that $X_\star$ is the unique stabilizing solution to the SCARE $\op C(X)=0$.
\end{proof}
\subsection{Low-rank expression of a special residual}\label{ssec:low-rank-expression-of-a-special-residual}

To design an RADI-type method, a critical challenge is to identify an effective initial approximation that is expected to be low-rank and to yield a low-rank residual.
Inspired by 
\cite[Theorem~3.2]{guoL2023stochastic}
and the convergence properties of fixed point iterations \cite[Theorem 2.1]{guoL2023stochastic}, we can compute an approximation for the unique stabilizing solution $X_\star$ by pursuing the following process:
\[
	X_{t+1}
	=E_{\gamma}^{\T} \ltimes X_{t} \ltimes 
	(I +G_{\gamma} \ltimes X_{t} )^{-1} \ltimes  E_{\gamma}+H_{\gamma}, \qquad X_0=0,
\]
where $E_\gamma\in \mathbb{R}^{rn\times n}, G_\gamma \in \mathbb{R}^{rn\times rn}$ are determined by $A,B,C,\gamma$ (their detailed forms are omitted as they are not needed later) and
\begin{equation}
	\label{eq:H:scare:real}
	H_\gamma =
	2\gamma A_\gamma^{-\alt{\!}\T}C^{\T}(I_l+Y_\gamma Y_\gamma^{\T})^{-1}CA_\gamma^{-1}\succeq 0 \in \mathbb{R}^{n\times n},
\end{equation}
for proper $\gamma>0$.	Here, $A_\gamma=A-\gamma I, Y_\gamma = CA_\gamma^{-1}B$. 

We aim to use 
$X= H_\gamma
$ as an approximate solution, where $
H_\gamma$ is as in \cref{
eq:H:scare:real}, and construct its associated new residual operator $\op C_{X}$ in terms of $\op C_{X}(0)=\op C(X)$. 
\begin{theorem}\label{thm:residual-positive-semi-definite}
	Let $A_\gamma=A-\gamma I,Y_\gamma = CA_\gamma^{-1}B,\what Y_\gamma = \sqrt{2\gamma}CA_\gamma^{-1}\ltimes\what B$.
For $X=2\gamma A_{\gamma}^{-\alt{\!}\T}C^{\T}(I+Y_{\gamma}Y_{\gamma}^{\T})^{-1}CA_{\gamma}^{-1}$, we have
	\begin{equation}\label{eq:residual-total}
		\begin{gathered}[t]
			\op C(X) = \wtd C^{\T}\wtd C,
			\\ 
			\wtd C= 
			\begin{bmatrix}
				C+{2\gamma}(I+Y_{\gamma}Y_{\gamma}^{\T})^{-1}CA_{\gamma}^{-1}
				\\
				\wtd M^{-1}\left(\sqrt{2\gamma}CA_{\gamma}^{-1}\ltimes \what A - \what Y_\gamma B^{\T}X\right)
			\end{bmatrix}\in \mathbb{R}^{rl\times n}
			,
		\end{gathered}
		\end{equation}
		where
			$\wtd M\wtd M^{\T}=I+Y_{\gamma}Y_{\gamma}^{\T}\otimes I_{r-1}+\what Y_\gamma\what Y_\gamma^{\T}$.
			Note that $l$ is the number of rows of $C$.
	Moreover, 
	\[
		\begin{aligned}[t]
		L^{\T}_{X}=(P^{\alt{\!}-\alt{\!}\T}_{X})\what Y_{\gamma}^{\T}\ltimes (I\alt{\!}+\alt{\!}Y_{\gamma}Y_{\gamma}^{\T})^{\alt{\!}-\alt{\!}1}\alt{\!}\ltimes \alt{\!}\left(\alt{\!}\sqrt{2\gamma}CA_{\gamma}^{\alt{\!}-\alt{\!}1}\ltimes \what A \alt{\!}-\alt{\!} \what Y_{\gamma} B^{\T}X\alt{\!}\right)\alt{\!}
		\alt{\\}
		+P_{X} B^{\T}X.
		\end{aligned}
	\]
\end{theorem}
\begin{proof}
	Straightforward calculations yield $L^{\T}_{X}$, so we will focus on $\op C(X)$.
	By definition,  we express $
	\op C(X) = \text{(I)}+\text{(II)}$, where 
\begin{align*}
	\text{(I)}&=
	C^{\T}C + A ^{\T} X+ X A,
	\\
	\text{(II)}&=
			 \ulub{\what A ^{\T}\alt{\!}\ltimes X\ltimes \what A 
			 \alt{\!}-\alt{\!}}{(\what A ^{\T}\alt{\!}\ltimes X \ltimes \what B \alt{\!}+\alt{\!}X B )}
			 {(\what B ^{\T}\alt{\!}\ltimes X \ltimes \what B \alt{\!}+\alt{\!}I )^{\alt{\!}-\alt{\!}1}\alt{\!}(*}{)^{\T}\alt{\!\!}}.\alt{\!\!\!\!}
\end{align*}
By $ A =A_{{\gamma}}+{\gamma} I$ and writing $
C_{\gamma}=\sqrt{2\gamma}CA_\gamma^{-1}$ or equivalently $C=\frac{1}{\sqrt{2\gamma}}C_{\gamma}A_{{\gamma}}$, 
	\begin{align}
		\text{(I)}&=
		\frac{1}{2\gamma}(A_{{\gamma}}^{\T}C_{\gamma}^{\T})(C_{\gamma}A_{{\gamma}})
		\ulub{+}
		{(A_{{\gamma}}^{\T}+\gamma I)C_{\gamma}^{\T}(I+Y_{\gamma}Y_{\gamma}^{\T})^{-1}C_{\gamma}}
		{+(*}{)^{\T}}
		\nonumber\\ &=
		\ubul{\left(\vphantom{\frac{C_{\gamma}A_{{\gamma}}}{\sqrt{2\gamma}}}*\right.\alt{\!\!}}{\left.\vphantom{\frac{C_{\gamma}A_{{\gamma}}}{\sqrt{2\gamma}}}\right)^{\T}}{
		\left(\frac{C_{\gamma}A_{{\gamma}}}{\sqrt{2\gamma}}+\sqrt{2\gamma}(I+Y_{\gamma}Y_{\gamma}^{\T})^{-1}C_{\gamma}\right)}{}
		-2\gamma C_{\gamma}^{\T}(I+Y_{\gamma}Y_{\gamma}^{\T})^{-2}C_{\gamma}+2\gamma C_{\gamma}^{\T}(I+Y_{\gamma}Y_{\gamma}^{\T})^{-1}C_{\gamma}
		\nonumber\\ &=
		\ubul{(*}{)^{\T}}{
		({C}+\sqrt{2\gamma}(I+Y_{\gamma}Y_{\gamma}^{\T})^{-1}C_{\gamma})}{
		+ XBB^{\T}X.}
	\label{eq:tmp1}
	\end{align}
As in \cref{eq:AB-Xi} we write $R_{X}=I+\what B^{\T}\ltimes X\ltimes \what B=I+\what Y_\gamma^{\T}\ltimes(I+Y_{\gamma}Y_{\gamma}^{\T})^{-1}\ltimes\what Y_\gamma$
for $C_{\gamma}\ltimes \what B=\sqrt{2\gamma}CA_\gamma^{-1}\ltimes\what B=\what Y_\gamma$.
Since $\wtd M\wtd M^{\T}=I+Y_{\gamma}Y_{\gamma}^{\T}\otimes I_{r-1}+\what Y_\gamma\what Y_\gamma^{\T}$, and
\begin{subequations}\label{eq:MM1}
	\begin{gather}
	\label{eq:M1}
	R^{-1}_{X}\what Y_\gamma^{\T}\ltimes (I+Y_{\gamma}Y_{\gamma}^{\T})^{-1}
	\clue{\cref{eq:easy}}{=} \what Y_\gamma^{\T}
		(\wtd M\wtd M^{\T})^{-1}
	,
	\\
	\label{eq:M}
	\begin{multlined}[b]
(I+Y_{\gamma}Y_{\gamma}^{\T})^{-1}\otimes I_{r-1}
\alt{\\}
\ulub{-}{(I+Y_{\gamma}Y_{\gamma}^{\T})^{-1}\ltimes \what Y_\gamma} R^{-1}_{X}{(*}{)^{\T}}
	\end{multlined}
		\clue{\cref{eq:smwf}}{=}
		(\wtd M\wtd M^{\T})^{-1}
			,
	\end{gather}
\end{subequations}
we have
	\begin{align}
	\text{(II)}
	&=
	(\what A ^{\T}\ltimes C_{\gamma}^{\T})\ltimes (I+Y_{\gamma}Y_{\gamma}^{\T})^{-1}\ltimes (C_{\gamma}\ltimes \what A )
	\nonumber\\ &\ulub{\qquad-}{\left((\what A ^{\T}\ltimes C_{\gamma}^{\T})\ltimes (I\alt{\!}+\alt{\!}Y_{\gamma}Y_{\gamma}^{\T})^{\alt{\!}-\alt{\!}1}\ltimes \what Y_\gamma\alt{\!}+\alt{\!}XB\right)}
	{R^{-1}_{X}(*}{)^{\T}}
	\nonumber\\
	&=
	(C_{\gamma}\ltimes \what A )^{\T}\ltimes (I+Y_{\gamma}Y_{\gamma}^{\T})^{-1}\ltimes (C_{\gamma}\ltimes \what A )
	\nonumber\\
	&\ubul{\qquad-(*}{)^{\T}\ltimes \what Y_\gamma R^{-1}_{X} \what Y_\gamma^{\T} \ltimes }{(I+Y_{\gamma}Y_{\gamma}^{\T})^{-1}\ltimes (C_{\gamma}\ltimes \what A )}{}
	\nonumber\\
	&\ulub{\qquad-}{XBR^{-1}_{X}\what Y_\gamma^{\T}\ltimes (I+Y_{\gamma}Y_{\gamma}^{\T})^{-1}\ltimes (C_{\gamma}\ltimes \what A )}{-(*)}{^{\T}
	\nonumber\\& \qquad-XBR^{-1}_{X}B^{\T}X}
	\nonumber\\
	&
	\ubul{
\clue{\cref{eq:MM1}}{=}
		(*}{)^{\alt{\!}\T}\alt{\!}(\wtd M\wtd M^{\T})^{\alt{\!}-\alt{\!}1}}{(C_{\gamma}\ltimes \what A \alt{\!}-\alt{\!}\what Y_\gamma B^{\T}\alt{\!}X)}{
		\alt{\!}-\alt{\!}X B B^{\T}\alt{\!}X.}
\label{eq:tmp2}
	\end{align}
Combining \cref{eq:tmp1,eq:tmp2} gives
\cref{eq:residual-total}.
\end{proof}

Clearly, in the case $r=1$, 
which corresponds to the classical CARE,
the residual \cref{eq:residual-total} reduces 
to the forms presented in \cite[Proposition 1]{bennerBKS2018radi} or \cite[Theorem 11]{guoL2023intrinsic}.  
Based on these forms and utilizing different choices of $\gamma$, we can derive the efficient RADI method for CAREs as outlined \cite{bennerBKS2018radi}. 
Using the same methodology, we are able to develop an RADI-type method for SCAREs, which is detailed in \cref{alg:an-radi-type-method-for-scares}.
\begin{algorithm}[ht]
	\caption{An RADI-type method for the SCARE \cref{eq:scare-standard-form:scare} (prototypical)}\label{alg:an-radi-type-method-for-scares}
	\begin{algorithmic}[1]
		\REQUIRE $A,B,C,\what A,\what B$ and a strategy to generate $\gamma$.
		\STATE $\Xi\leftarrow [\;]$ (empty matrix).
		\LOOP
		\STATE Generate a proper $\gamma\in \mathbb{R}$.
		\STATE $C_\gamma=\sqrt{2\gamma}C(A-\gamma I)^{-1}, Y=\frac{1}{\sqrt{2\gamma}}C_\gamma B, \what Y=C_\gamma\ltimes \what B$.
		\STATE (Cholesky factorization) $NN^{\T}=I+Y Y^{\T}$.
		\STATE $\Xi\leftarrow\begin{bmatrix}
			\Xi & (N^{-1}C_\gamma)^{\T}
	\end{bmatrix}$.
		\STATE (Cholesky f.) $MM^{\T}=(I+Y Y^{\T})\otimes I_{r-1}+\what Y \what Y ^{\T}$.
		\STATE $C_M=C_{\gamma}\ltimes \what A -\sqrt{2\gamma}\what YY^{\T}N^{-\alt{\!}\T}N^{-1}C_\gamma$.
		\STATE $C\leftarrow
		\begin{bmatrix}
			C+\sqrt{2\gamma}N^{-\alt{\!}\T}N^{-1}C_\gamma
			\\
			M^{-1}C_M
		\end{bmatrix}$.
		\IF{$C$ satisfies some stopping criterion}
		\RETURN
		\ENDIF
		\STATE (Cholesky f.) $K^{\T}K=I+(N^{-1}\ltimes \what Y)^{\T}(N^{-1}\ltimes \what Y)$.
		\STATE $
		L^{\T}\alt{\!\!}=\alt{\!}K^{\alt{\!}-\alt{\!\!}\T}\alt{\!}(N^{\alt{\!}-\alt{\!}1}\alt{\!}\ltimes \what Y)^{\alt{\!}\T}\alt{\!}\ltimes N^{\alt{\!}-\alt{\!}1}\alt{\!}\ltimes C_M
		\alt{\!}+\alt{\!}\sqrt{2\gamma}KY^{\T}\alt{\!}N^{\alt{\!}-\alt{\!\!}\T}\alt{\!}N^{\alt{\!}-\alt{\!}1}\alt{\!}C_\gamma.\alt{\!}
	$
		\STATE $B\leftarrow BK^{-1},\what B\leftarrow \what BK^{-1}$.
		\STATE $A\leftarrow A-BL^{\T},\what A\leftarrow \what A-\what BL^{\T}$.
		\ENDLOOP
		\ENSURE $\Xi$ that satisfies $X_\star\approx \Xi\Xi^{\T}$.
	\end{algorithmic}
\end{algorithm}

\section{Implementation aspects}\label{sec:implementation-aspects}
In this section, we discuss several techniques aimed at reducing both storage and computational costs. 
Note that $l\ll n,m\ll n, r\ll n$ and both $A,\what A$ are sparse.

\subsection{Storage and compression}\label{ssec:storage-reduction}

Upon careful examination of the sizes of matrices involved in \cref{alg:an-radi-type-method-for-scares}, 
we present the following table, where
for instance $\what A$ locating in the $(1,1)$ block implies $\what A\in \mathbb{R}^{(r-1)n\times n}$, and any term changing its meaning in one iteration is distinguished with the subscript ``$\cdot_{\text{new}}$'':
\[
\begin{array}{c|c|c|c|c}
                  & n                            & m                            & (r-1)l & l \\ 
(r-1)n            & \what A,\what A_{\text{new}} & \what B,\what B_{\text{new}} &        &   \\
n                 & A,
A_{\text{new}}    & B,B_{\text{new}}             &        &   \\
m                 & L^{\T},F                          & 
K                          &        &   \\
rl=l_{\text{new}} & C_{\text{new}}               &                              &        &   \\
(r-1)l            &                              & \what Y
& M      &   \\
l                 & C,C_\gamma
& Y
&        & 
N \\
\end{array}
\]

To clarify the terms that appear during the iterations in \cref{alg:an-radi-type-method-for-scares},
for any term $Z$ in the $k$-th iteration for $k=0,1,2,\dots$, we denote its first appearance by $Z_{(k)}$, and $Z_{(k+1)}$ if it is updated in the iteration (luckily each term is updated at most once in each iteration).
Then $l_{(k)}=r^kl$.

The primary storage cost is concentrated in $A_{(k)},\what A_{(k)}$ having $rn^2$ elements in total, and $C_{(k)}$ having $l_{(k)}n=r^kln$ elements that 
increases exponentially.

For the former, note that 
\[
		A_{(k+1)}\alt{\!}=\alt{\!}A_{(0)}\alt{\!}-\alt{\!}B_{(1)}L^{\T}_{(0)}\alt{\!}-\alt{\!}\dots\alt{\!}-\alt{\!}B_{(k+1)}L^{\T}_{(k)}
		\alt{\!}=\alt{\!}A_{(0)}\alt{\!}+\alt{\!}B_{(0)}F_{(k)},
\]
where
$F_{(k)}\alt{\!}=\alt{\!}-\alt{\!}\left(K_{(0)}^{-1}L^{\T}_{(0)}\alt{\!}+\alt{\!}\dotsb\alt{\!}+\alt{\!}(K_{(0)}^{-1}\dotsm K_{(k)}^{-1})L^{\T}_{(k)}\right)$.
Thus, instead of updating $A_{(k)}$, we can focus solely on updating $F_{(k)}$.
Similarly, throughout the entire process, $A,B,\what A,\what B$ can be maintained without any modification.

For the latter,
compression will be employed to manage the storage costs.
In detail, when the number of rows of $C$ exceeds a certain threshold, 
we utilize its truncated singular value decomposition instead.
Suppose 
\[
	C=\begin{bmatrix}
	 	U & U^{\perp}
	 \end{bmatrix}\begin{bmatrix}
	 \Sigma &\\ &\Sigma^{\perp}
	 \end{bmatrix}\begin{bmatrix}
	 	V & V^{\perp}
	 \end{bmatrix}^{\T}=U\Sigma V^{\T}+\Omega.
\]
Recall \cref{thm:residual-positive-semi-definite} and treat $X=X(C),\wtd C=\wtd C(C)$ as functions of $C$.
Let $C^{\trun}:=\Sigma V^{\T}$ serve as a replacement for $C$ during the calculation process. 
Write $X^{\trun}=X(C^{\trun}),\wtd C^{\trun}=\wtd C(C^{\trun})$, and $\op C^{\trun}(\cdot)=\op C(\cdot)-C^{\T}C+(C^{\trun})^{\T}C^{\trun}=\op C(\cdot)-\Omega^{\T}\Omega$.
Then $X^{\trun}$, as an approximate solution to $\op C^{\trun}(X)=0$, satisfies $\op C^{\trun}(X^{\trun})=(\wtd C^{\trun})^{\T}\wtd C^{\trun}$ by \cref{thm:residual-positive-semi-definite}.
On the other hand, $X^{\trun}$ is still an approximate solution to $\op C(X)=0$, which satisfies
$\op C(X^{\trun})=(\wtd C^{\trun})^{\T}\wtd C^{\trun}+\Omega^{\T}\Omega$.
By truncating $C_{(k)}$ during each iteration, we obtain
\[
	\op C(X^{\trun}_{(k)})={C^{\trun}_{(k)}}^{\alt{\!\!}\T}C^{\trun}_{(k)}+\Omega_{(k)}^{\T}\Omega_{(k)}
	+\dots+\Omega_{(0)}^{\T}\Omega_{(0)}
	.
\]
In many scenarios, with a proper truncation criterion, the number of rows of $C_{(k)}$ is uniformly upper bounded by a fixed integer $l_{\trun}$,
which keeps the storage limited.

\subsection{Time complexity reduction}\label{ssec:complexity}

	Estimating the complexity of multiplying or solving linear systems with sparse matrices is intricate and varies based on the applied algorithms. 
	To describe it, let $\sol_p(M)$ be the number of floating-point operations (flops) required for computing $x^{\T}(M-\lambda I)^{-1}$ with a sparse matrix $M\in \mathbb{R}^{p\times p}$ and a scalar $\lambda\in \mathbb{R}$ by using a sparse matrix solver on a vector $x$,
and $\mul_{p\times q}(M)$ be that for $x^{\T}M$ with a (sparse) matrix $M\in \mathbb{R}^{p\times q}$ and a (dense) vector $x\in \mathbb{R}^p$.
Clearly, $\mul_{p\times q}(M)=(2p-1)q$ if $M$ is dense.
Moreover, for $M=\alt{\!}\begin{bmatrix}
	M_1 \\ \vdots \\ M_r
\end{bmatrix}\alt{\!}$ with $M_i\in \mathbb{R}^{p\times q}$ and $x\in \mathbb{R}^p$, $x^{\T}\ltimes M=\alt{\!}\begin{bmatrix}
	x^{\T}M_1 \\ \vdots \\ x^{\T}M_r
\end{bmatrix}\alt{\!}$ needs
$\sum_{i=1}^r\mul_{p\times q}(M_i)=:\mul^r_{p\times q}(M)$ flops.
In many scenarios, $\sol_p(M)$, $\mul_{p\times q}(M)$, and $\mul^r_{p\times q}(M)$ are related to the number $\nnz(M)$ of nonzero entries in $M$, but not necessarily.

Considering \cref{alg:an-radi-type-method-for-scares}, in the direct calculations,
the main time complexity concentrates in computing $C_{\gamma_{(k)}}$.
A naive inversion (solving the linear equations) needs $\frac{2}{3}n^3l_{(k)}$ flops (or $\frac{2}{3}n^3l_{\trun}$ after compression), which dominates 
the overall computational cost. 
However, leveraging techniques similar to those in \cite{bennerLP2008numerical,bennerBKS2018radi},
$C_{\gamma_{(k)}}$ can be obtained with significantly fewer than $\frac{2}{3}l_{(k)}n^3$ flops,
for $A_{(0)}=A$ is sparse and $A_{(k)}=A+BF_{(k-1)}$.


Write $A_{\gamma_{(k)}}=A-\gamma_{(k)}I$, and
then
\begin{align*}
	\alt{\MoveEqLeft[0]} C_{(k)}(A_{(k)}-\gamma_{(k)}I)^{-1}
	\alt{\\}&= C_{(k)}(A_{\gamma_{(k)}}+BF_{(k-1)})^{-1}
	\\&\clue{\cref{eq:smwf}}{=} 
	C_{(k)}\alt{\!}\left(A_{\gamma_{(k)}}^{-1}-A_{\gamma_{(k)}}^{-1}B(I+F_{(k-1)}A_{\gamma_{(k)}}^{-1}B)^{-1}F_{(k-1)}A_{\gamma_{(k)}}^{-1}\right)\alt{\!}
	\\&= C_{(k)}A_{\gamma_{(k)}}^{-1}\alt{\!}-\alt{\!}\ul{C_{(k)}A_{\gamma_{(k)}}^{-1}}\alt{\!}B(I\alt{\!}+\alt{\!}F_{(k-1)}A_{\gamma_{(k)}}^{-1}\alt{\!}B)^{-1}\ul{F_{(k-1)}A_{\gamma_{(k)}}^{-1}}, 
\end{align*}
which can be computed with complexity
	\begin{align}
		&\ub{(l_{(k)}+m)\sol_n(A)}_{[C_{(k)}, F_{(k-1)}]A_{\gamma_{(k)}}^{-1}}+\ub{(l_{(k)}+m)\mul_{n\times m}(B)}_{[C_{(k)}, F_{(k-1)}]A_{\gamma_{(k)}}^{-1}B}
		+\ub{\tfrac{2}{3}m^3+m^2l_{(k)}+m(m-1)l_{(k)}}_{*(I+*)^{-1}}+\ub{l_{(k)}(2m-1)n+l_{(k)}n}_{*-*\times *}
		\nonumber\\	\approx&\;	
		(l_{(k)}+m)\left(\sol_n(A)+\mul_{n\times m}(B)\right)+2l_{(k)}mn,
\label{eq:complexity:CAinv}
	\end{align}
or $(l_{\trun}+m)\left(\sol_n(A)+\mul_{n\times m}(B)\right)+2l_{\trun}mn$ after compression.
Here ``$\approx$'' indicates that the lower order terms are ignored, considering $m\ll n,r\ll n$.
Moreover,
\begin{align*}
	Y_{(k)}&=C_{(k)}(A_{(k)}-\gamma_{(k)}I)^{-1}B_{(k)}
	\\&\clue{\cref{eq:smwf}}{=} 
	C_{(k)}\alt{\!}\left(\alt{\!}A_{\gamma_{(k)}}^{-1}-A_{\gamma_{(k)}}^{-1}B(I\alt{\!}+\alt{\!}F_{(k\alt{\!}-\alt{\!}1)}A_{\gamma_{(k)}}^{-1}B)^{\alt{\!}-\alt{\!}1}F_{(k\alt{\!}-\alt{\!}1)}A_{\gamma_{(k)}}^{-1}\alt{\!}\right)\alt{\!}
	\alt{\\&\quad\;\,\qquad\qquad\qquad\qquad\qquad\qquad\qquad\cdot} BK_{(0)}^{-1}\dotsm K_{(k-1)}^{-1}
	\\&=
	\alt{\!}\left(\alt{\!}C_{(k)}A_{\gamma_{(k)}}^{-1}B(I\alt{\!}+\alt{\!}F_{(k-1)}A_{\gamma_{(k)}}^{-1}B)^{-1}\alt{\!}\right)\alt{\!}\left(\alt{\!}K_{(0)}^{-1}\dotsm K_{(k-1)}^{-1}\alt{\!}\right)\alt{\!},
\end{align*}
consisting of two factors already in hand.

The implicit calculations in each iteration of \cref{alg:an-radi-type-method-for-scares} include generating $\gamma$ in the beginning and verifying the stopping criterion at the end.
%
Generating appropriate values for $\gamma$ necessitates a detailed discussion, which is covered
in \cref{sssec:shift-strategy}.

To verify the stopping criterion,
let $\N{\cdot}$ be a unitary-invariant norm.
Then the approximate accuracy is measured by
\[
	\nres(X_{(k)}):=\frac{\N{\op C(X_{(k)})}}{\N{\op C(0)}}=\frac{\N{C_{(k)}^{\T} C_{(k)}}}{\N{C^{\T}C}}=\frac{\N{C_{(k)} C_{(k)}^{\T}}}{\N{CC^{\T}}}
	,
\]
whose complexity is $l_{(k)}^2\alt{\!}(2n\alt{\!}-\alt{\!}1)+\OO(l_{(k)}^3)\alt{\!}=\alt{\!}2r^{2k}l^2n+\OO(r^{3k}l^3)$. 
In particular, if the norm is chosen as the trace norm $\N{\cdot}_*$, then 
\[
	\nres_*(X_{(k)}):=\frac{\N{C_{(k)} C_{(k)}^{\T}}_*}{\N{CC^{\T}}_*}
	=\frac{\N{ C_{(k)}}_F^2}{\N{C}_F^2},
\]
and the complexity is reduced to $2l_{(k)}n=2r^{k}ln$.
Here $\N{\cdot}_F$ is the Frobenius norm.

When compression is applied, it then holds 
\begin{align*}
	\alt{&}\nres(X_{(k)}^{\trun})
	\alt{\\}&=\alt{\!}\frac{\N{\op C(X_{(k)}^{\trun})}}{\N{\op C(0)}}\alt{\!}=\alt{\!}\frac{\N{{C^{\trun}_{(k)}}^{\alt{\!\!}\T} C^{\trun}_{(k)}\alt{\!}+\alt{\!}\sum_{i=0}^{k}\alt{\!} \Omega_{(i)}^{\T}\Omega_{(i)}}}{\N{C^{\T}C}}
	\\&\le\alt{\!} \frac{\N{C^{\trun}_{(k)} {C^{\trun}_{(k)}}{}^{\alt{\!\!}\T}}}{\N{CC^{\T}}}\alt{\!}+\alt{\!\!}\sum_{i=0}^{k}\alt{\!}\frac{\N{\Omega_{(i)} \Omega_{(i)}^{\T}}}{\N{CC^{\T}}}
	\alt{\!}=\alt{\!} \frac{\N{\Sigma_{(k)}^2}}{\N{CC^{\T}}}\alt{\!}+\alt{\!\!}\sum_{i=0}^{k}\alt{\!}\frac{\N{(\Sigma_{(i)}^{\perp})^2}}{\N{CC^{\T}}}
		.
\end{align*}
The complexity is thus governed by the compression process, which relies on the spectral factorization of $C_{(k)}C_{(k)}^{\T}$:
\begin{align*}
	&\ub{l_{(k)}^2\alt{\!}(2n\alt{\!}-\alt{\!}1)}_{CC^{\T}}\alt{\!}+\alt{\!}\ub{\OO(l_{(k)}^3)}_{\makebox[0pt]{\scriptsize $U\Sigma\Sigma^{\T}U^{\T}=CC^{\T}$}}
	\alt{\!}+\alt{\!}\ub{(2l_{(k)}\alt{\!\!}-\alt{\!}1)l_{(k)}n}_{\Sigma V^{\T}=U^{\T}C}
	\alt{\!}\approx\alt{\!} 4r^2l_{\trun}^2n,
	\text{for $l_{(k)}\alt{\!\!}=\alt{\!}rl_{\trun}$.}
\end{align*}
This so-called cross-product algorithm for singular value decomposition implies the cost of compression is a lower-order term compare with that of one iteration (see \cref{ssec:complexity-list}).
If the cross-product-free algorithms are employed, the complexity becomes up to $6r^2l_{\trun}^2n+11r^3l_{\trun}^3$ (see, e.g. \cite{golubL1996matrix}), which is still a lower-order term.

\subsection{Practical algorithm and its complexity}\label{ssec:complexity-list}
Considering the details in \cref{ssec:storage-reduction,ssec:complexity}, we propose a practical method for SCAREs, namely \cref{alg:an-radi-type-method-for-scares:sparse},
in which the underlined terms have been calculated out in the indicated line or the same line.
Note that \cref{alg:an-radi-type-method-for-scares:sparse} does not reduce 
to the RADI method even for the CAREs, because RADI does not require a compression step, which is indeed necessary for SCAREs.

Additionally, The matrix $E$ in the input is set as $I$ and is therefore not utilized in this context; however, it would prove beneficial for solving the so-called generalized SCAREs, which will be discussed in \cref{ssec:generalized-SCARE}.

\begin{algorithm}[ht]
	\caption{an RADI-type method for the SCARE \cref{eq:scare-standard-form:scare} (practical)}\label{alg:an-radi-type-method-for-scares:sparse}
	\begin{algorithmic}[1]
		\REQUIRE $A,B,C,\what A,\what B$ and a strategy to generate $\gamma$, $E=I_n$.
		\STATE $F=0,C_M=0,K_{\Pi}=I,\Xi=[\;]$, $\nu_0=\N{C}_F^2$, $\nu_\Omega=0$.
		\LOOP
		\STATE Generate a proper $\gamma\in \mathbb{R}$.
		\STATE $C_A=C(A-\gamma E)^{-1},F_A=F(A-\gamma E)^{-1}$, $C_\gamma=\sqrt{2\gamma}(C_A-C_AB(I+F_AB)^{-1}F_A)$.
		\STATE $Y\alt{\!}=\alt{\!}\ul{C_AB(I\alt{\!}+\alt{\!}F_AB)^{\alt{\!}-\alt{\!}1}}_4K_{\Pi}^{\alt{\!}-\alt{\!}1}$, $C_M\alt{\!}\leftarrow\alt{\!} C_\gamma\ltimes \what A$, $\what Y\alt{\!}\leftarrow\alt{\!} C_\gamma\ltimes \what B.\alt{\!}$\alt{\!\!\!\!}
		\STATE (Cholesky factorization) $NN^{\T}=I+Y Y^{\T}$.
		\STATE $S=N^{-1}C_\gamma$, $\Xi\leftarrow\begin{bmatrix}
				\Xi & S^{\T}
	\end{bmatrix}$.
	\STATE $C\alt{\!}\leftarrow\alt{\!} C\alt{\!}+\alt{\!\!}\sqrt{2\gamma}N^{\alt{\!}-\alt{\!\!}\T}\alt{\!}SE$, $F\alt{\!}\leftarrow\alt{\!} F\alt{\!}-\alt{\!}\left(\alt{\!}K_{\Pi}^{\alt{\!}-\alt{\!}1}Y^{\T}\right)\alt{\!\!}\ul{(\alt{\!}\sqrt{2\gamma}N^{\alt{\!}-\alt{\!\!}\T}\alt{\!}SE)}.\alt{\!}$
		\STATE $C_M\leftarrow C_M+\what YF,\what Y\leftarrow \what YK_{\Pi}^{-1}$.
		\STATE (Cholesky f.) $K^{\T}K=I+\what Y^{\T}(N^{-\alt{\!}\T}\ltimes (N^{-1}\ltimes \what Y))$.
		\STATE $K_{\Pi}\leftarrow KK_{\Pi}$.
		\STATE $F\leftarrow F-\left(K_{\Pi}^{-1}K^{-\alt{\!}\T}\ul{(N^{-\alt{\!}\T}\ltimes (N^{-1}\ltimes \what Y))}_{10}^{\T}\right)  C_M$.
		\STATE (Cholesky f.) $MM^{\T}=\ul{(I+YY^{\T})}_6\otimes I_{r-1}+\what Y \what Y ^{\T}$.
		\STATE (Truncated SVD) $\begin{bmatrix}
			C
			\\
			M^{-1}C_M
		\end{bmatrix}= (*)\Sigma V^{\T}+\Omega$ under some truncation criterion.
		\STATE  $\nu_\Omega\leftarrow \nu_\Omega+\N{\Omega}_F^2$, $C\leftarrow \Sigma V^{\T}$.
		\IF{$\N{\Sigma}_F^2+\nu_\Omega\le \nu_0\varepsilon $ (stopping criterion)}
		\RETURN
		\ENDIF
		\ENDLOOP
		\ENSURE $\Xi$ that satisfies $X_\star\approx \Xi\Xi^{\T}$.
	\end{algorithmic}
\end{algorithm}

Similarly we have the sizes of matrices appearing in \cref{alg:an-radi-type-method-for-scares:sparse} listed in the table below:
\[
	\begin{array}{@{}c|c|c|c|c@{}}
                   & n                                   & m                  & (r-1)l_{\trun} & l_{\trun} \\ 
	(r-1)n         & \what A                             & \what B            &                &           \\
	n              & A,E                                   & B                  &                &           \\
	m              & F,F_A & K_{\Pi}, K &                &           \\
	rl_{\trun}     & \Omega                    &                    &                &           \\
	(r-1)l_{\trun} & C_M                                 & \what Y            & M              &           \\
	l_{\trun}      & C,C_A,C_\gamma,S,V^{\T}  & Y                  &          & N,\Sigma        \\
\end{array}
\]
Note that in the initial iteration when the truncation of $C$ is not performed, $l_{\trun}$ in the table should be replaced by $l$.
With the table we can easily count the space and time complexity in one iteration as follows,
where ``$\approx$'' indicates that the lower order terms are ignored, considering $l_{\trun}\ll n, m\ll n,r\ll n$.

Storage requirement:
\begin{itemize}
	\item $A,\what A,B,\what B$ remain unchanged throughout the process, requiring no extra storage.
	\item $S$, as part of the output, is stored in the output, while its storage can be allocated for $\sqrt{2\gamma}N^{-\alt{\!}\T}S,V$ in the last iteration and $C_A,C_\gamma$ in the current iteration before $S$ occupies it.
	\item $F$ requires $mn$ units, and $F_A$ requires $mn$ units.
	\item $C$ and $C_M$ require $rl_{\trun}n$ units that is also shared by $\Omega$.
	\item $K_{\Pi},K,\what Y,Y,M,N,\Sigma$ require $\ll\alt{\!} n$ units respectively and 
		can be safely ignored.
\end{itemize}
To sum up, besides the storage of $tl_{\trun}n$ units for the output ($t$ is the number of iterations), the intermediate terms still require $\approx(2m+rl_{\trun})n$ units.

Time complexity:
\begin{itemize}
	\item Step~3: will be discussed in \cref{sssec:shift-strategy}.
	\item Step~4:
		$
		\begin{multlined}[t]
			\alt{\!\!}\approx\alt{\!}\ub{(l_{\trun}\alt{\!}+\alt{\!}m)\alt{\!}\left(\sol_n\alt{\!}(A)\alt{\!}+\alt{\!}\mul_{n\alt{\!}\times\alt{\!} m}\alt{\!}(B)\right)\alt{\!}+\alt{\!}2l_{\trun}mn}_{\text{\cref{eq:complexity:CAinv}}}
			\alt{\!}+\alt{\!}\ub{\alt{\!}l_{\trun}n}_{\makebox[6pt][r]{\footnotesize $\sqrt{2\gamma}$}*}
			\alt{\!}.\alt{\!\!}
		\end{multlined}
		$
	\item Step~5: 
		$\ub{l_{\trun}m^2}_{Y}+\ub{l_{\trun}\mul^{r-1}_{n\times n}(\what A)}_{C_M}+\ub{l_{\trun}\mul^{r-1}_{n\times m}(\what B)}_{\what Y}
		$.
	\item Step~6: $\ub{l_{\trun}^2(2m-1)+l_{\trun}}_{I+YY^{\T}}+\ub{\tfrac{1}{3}l_{\trun}^3}_{N}=\oo(n)$.
	\item Steps~7--8: $
		\begin{aligned}[t]
			&\ub{l_{\trun}^2n}_{S}+\ub{l_{\trun}^2n+l_{\trun}n}_{\sqrt{2\gamma}N^{-\alt{\!}\T}S}+\ub{l_{\trun}n}_{C}
		\alt{\\[-3pt]&\qquad\qquad\quad}+\ub{l_{\trun}m^2+(2l_{\trun}-1)mn}_{(K_{\Pi}^{-1}Y^{\T})(*)}+\ub{mn}_{F}.
		\end{aligned}
		$
	\item Step~9: $
		\begin{aligned}[t]
			&\ub{(r-1)l_{\trun}(2m-1)n+(r-1)l_{\trun}n}_{C_M}
			\alt{\\[-3pt]&\qquad\qquad\qquad}
			+\ub{(r-1)l_{\trun}m^2}_{\what Y}
			\approx 2(r-1)l_{\trun}mn.
		\end{aligned}
		$
	\item Step~10: $
		\begin{aligned}[t]
			&\ub{(r-1)l_{\trun}^2m}_{N^{-1}\ltimes \what Y}+\ub{(r-1)l_{\trun}^2m}_{N^{-\alt{\!}\T}\ltimes *}
			\alt{\\[-3pt]&\quad}+\ub{(2(r-1)l_{\trun}-1)m^2+m}_{I+\what Y^{\T}*}+\ub{\tfrac{1}{3}m^3}_{K}=\oo(n).
		\end{aligned}
		$
	\item Steps~11--12: $
		\begin{aligned}[t]
			&\ub{\tfrac{m(m+1)(2m+1)}{6}}_{K_{\Pi}}+\ub{2(r-1)l_{\trun}m^2}_{K_{\Pi}^{-1}K^{-\alt{\!}\T}(*)}
			\alt{\\[-3pt]&\qquad\;\qquad}+\ub{(2(r-1)l_{\trun}-1)mn}_{* C_M}+\ub{mn}_{F}.
		\end{aligned}
		$
	\item Step~13: $
		\begin{aligned}[t]
			\ub{(r-1)^2l_{\trun}^2(2m-1)}_{\what Y\what Y^{\T}}+\ub{(r-1)l_{\trun}^2}_{MM^{\T}}+\ub{\tfrac{1}{3}(r-1)^3l_{\trun}^3}_{M}
			\alt{\\[-3pt]}
			=\oo(n).
		\end{aligned}
		$
	\item Steps~14--18: $\ub{(r-1)^2l_{\trun}^2n}_{M^{-1}C_M}+\ub{6r^2l_{\trun}^2n+11r^3l^3}_{\text{SVD}}
		$.
\end{itemize}
The total complexity is  
 $\approx (l_{\trun}+m)(\sol_n(A)+\mul_{n\times m}(B))+l_{\trun}(\mul^{r-1}_{n\times n}(\what A)+\mul^{r-1}_{n\times m}(\what B))+([7r^2-2r+3]l_{\trun}+4rm+3)l_{\trun}n$.
In particular, if $B$ and $\what B$ are not sparse, then it is 
$(l_{\trun}+m)\sol_n(A)+2l_{\trun}\mul^{r-1}_{n\times n}(\what A)+([7r^2-2r+3]l_{\trun}+6rm+3)l_{\trun}n+2m^2n$.

\subsection{Shift strategy}\label{sssec:shift-strategy}
As we will see, the convergence of the proposed method heavily depends on the strategy of choosing the shifts $\gamma$.
Since the method is of the RADI-type for SCAREs, it is logical to adopt the effective shift strategy of the RADI method for CAREs.
However, since it remains unclear whether in two successive iterations using $\gamma\in \mathbb{C}$ and its conjugate $\ol\gamma$ is equivalent to using $\ol \gamma$ and $\gamma$ due to the appearance of the left semi-tensor product,
to sidestep the potential trouble that the complex version would perhaps not be decomplexificated,
we are self-restricted to use real shifts. 

One is the so-called residual Hamiltonian shift \cite{bennerBKS2018radi}.
The idea is to project the SCARE at the current iteration onto a small-dimensional subspace spanned by the approximate solutions in several previous iterations,
and then find a suitable shift based on the projected SCARE.
Equivalently, for $\gamma_{(k)}$, letting $U$ be the orthonormal basis of the subspace $\subspan\set{S_{(k-1)},S_{(k-2)},\dots,S_{(k-s)}}$ for a prescribed $s$, then the projected SCARE is 
	\begin{stripalign}
\begin{align}
		\nonumber
		&U^{\T}C_{(k)}^{\T}C_{(k)}U+ U^{\T}A_{(k)}^{\T}U\breve{X}+\breve{X} U^{\T}A_{(k)}U 
		+ U^{\T}\ltimes\what A_{(k)}^{\T}\ltimes U\ltimes \breve{X}\ltimes U^{\T}\ltimes\what A_{(k)}\ltimes U \\
		&\qquad\qquad\ubul{=(*}{)^{\T}(I+ \what B_{(k)}^{\T}\ltimes U\ltimes \breve{X} \ltimes U^{\T}\ltimes \what B_{(k)})^{-1}}{(B_{(k)}^{\T}U\breve{X}+\what B_{(k)}^{\T}\ltimes U\ltimes \breve{X} \ltimes U^{\T}\ltimes \what A_{(k)}\ltimes U)}{.}
		\nonumber
\end{align}
	\end{stripalign}
Since solving this remains challenging, we compute the corresponding CARE instead:
\[
	U^{\T}\alt{\!}C_{(k)}^{\T}C_{(k)}U\alt{\!}+U^{\T}\alt{\!}A_{(k)}^{\T}U\breve{X}\alt{\!}+\alt{\!}\breve{X} U^{\T}\alt{\!}A_{(k)}U
	\alt{\!}=\alt{\!}\breve{X}U^{\T}\alt{\!}B_{(k)}B_{(k)}^{\T}\alt{\!}U\breve{X}
	,\alt{\!}
\]
which is obtained by treating $X$ in its products with $\what A$ or $\what B$ as $X_{(k-1)}$,
an approximation of the projected SCARE.

The shift is computed as follows:
determine all the eigenpairs $\left(\lambda,\begin{bmatrix}
	\breve{r}\\ \breve{q}
\end{bmatrix}\right)$ of the associated Hamiltonian matrix 
$
	\begin{bmatrix}
U^{\T}A_{(k)}U &U^{\T}B_{(k)}B_{(k)}^{\T}U\\
U^{\T}C_{(k)}^{\T}C_{(k)}U&-U^{\T}A_{(k)}^{\T}U
\end{bmatrix}\alt{\!},
$
and then select the eigenvalue $\lambda$ whose associated eigenvector satisfies that the corresponding $\N{\breve{q}}$ is the largest among all the eigenvectors.
Since we only use real shifts, $\gamma_{(k)}$ is chosen as the real part of $\lambda$.

Another approach is the so-called projection shift related to the Krylov subspace method for CARE.
Here, the shift $\gamma_{(k)}$ is chosen as the real part of the eigenvalue of the projected matrix $U^{\T}A_{(k)}U$ that has the smallest real part.
In fact, $A_{(k)}$ is the approximate closed-loop system matrix.

Concerning the complexity, since that for the projection shift is part of that for the residual Hamiltonian shift, we only focus on the residual Hamiltonian shift.
The way for computing all terms is listed as follows, with the time complexity included: 
\begin{itemize}
	\item use QR factorization to compute $U\in \mathbb{R}^{n\times sl_{\trun}}$: $2(2s^2l_{\trun}^2n-\frac{2}{3}s^3l_{\trun}^3)\approx 4s^2l_{\trun}^2n$ if Householder orthogonalization is used or its half if modified Gram-Schmidt process is used.
	\item get $U^{\T}\alt{\!}B_{(k)}B_{(k)}^{\T}\alt{\!}U$ by $U^{\T}\alt{\!}B_{(k)}\alt{\!}=\alt{\!}(U^{\T}\alt{\!}B)(K_{(0)}^{\alt{\!}-\alt{\!}1}\dotsm K_{(k\alt{\!}-\alt{\!}1)}^{\alt{\!}-\alt{\!}1})$:
		\[
			\ub{sl_{\trun}\mul_{n\times m}(B)}_{U^{\T}B}+\ub{sl_{\trun}m^2}_{U^{\T}BK_{\Pi}^{-1}}+\ub{s^2l_{\trun}^2(2m-1)}_{*^{\T}*}
			.
		\]
	\item compute $U^{\T}C_{(k)}^{\T}C_{(k)}U$:
		$\ub{sl_{\trun}^2(2n-1)}_{C_{(k)}U}+\ub{s^2l_{\trun}^2(2l-1)}_{*^{\T}*}
		$.
	\item get $U^{\T}\alt{\!}A_{(k)}U$ by $U^{\T}\alt{\!}A_{(k)}U\alt{\!}=\alt{\!}U^{\T}\alt{\!}AU\alt{\!}-\alt{\!}(U^{\T}\alt{\!}B)(F_{(k)}U)$:
		\[
			\begin{multlined}[t]
			\ub{sl_{\trun}\mul_{n\times n}(A)}_{AU}+\ub{s^2l_{\trun}^2(2n-1)}_{U^{\T}AU}+\ub{sl_{\trun}m(2n-1)}_{F_{(k)}U}
			\alt{\\}
			+\ub{s^2l_{\trun}^2(2m-1)}_{U^{\T}BF_{(k)}U}+\ub{s^2l_{\trun}^2}_{*-*}
				.
			\end{multlined}
		\]
\end{itemize}
The total complexity is $\approx  sl_{\trun}(\mul_{n\times n}(A)+\mul_{n\times m}(B))+sl_{\trun}(2(s+1)l_{\trun}+2m)n$ or $\approx  sl_{\trun}\mul_{n\times n}(A)+sl_{\trun}(2(s+1)l_{\trun}+4m)n$ if $B$ is not sparse.
The storage requirement mainly consists of $2sl_{\trun}n$ units for $U$ and $AU$.

Comparing this complexity to that in other parts in \cref{alg:an-radi-type-method-for-scares:sparse}, 
we observe that shift generation is non-negligible.
To reduce complexity, we could utilize all the possible eigenvalues of the respective matrices as the shifts in the next many iterations to avoid the calculation on shifts in every iteration.

\subsection{Deal with the original SCARE \texorpdfstring{\cref{eq:scare}}{(1.1)}}\label{ssec:deal-with-the-original-scare-eq}
Combining the relation between the original SCARE \cref{eq:scare} and the standard-form SCARE \cref{eq:scare-standard-form:scare},
\cref{alg:an-radi-type-method-for-scares:sparse} can be applied to solve \cref{eq:scare} after minor modifications, as outlined below:
\begin{itemize}
	\item The inputs are $A,B,\what A\leftarrow \begin{bmatrix}
			A_1\\ \vdots\\ A_{r-1}
	\end{bmatrix},\what B\leftarrow \begin{bmatrix}
	B_1\\ \vdots\\ B_{r-1}
\end{bmatrix}$ and $C$, where $C^{\T}C=Q-LR^{-1}L^{\T}$ and $C$ has $l\ll n$ rows.
	\item The initialization in Line~1 is:  $F=-R^{-1}L^{\T}$, $K_{\Pi}$ satisfies $K_{\Pi}^{\T}K_{\Pi}=R$.
	\item any Kronecker product $*\otimes I$ is replaced by $I\otimes *$, including those in the left semi-tensor products.
\end{itemize}

\subsection{Deal with the generalized SCARE}\label{ssec:generalized-SCARE}
In many practical applications, the control system is given as
\[
	E\alt{\!}\diff x(t) 
	\alt{\!}=\alt{\!} (A_{0}x(t) + B_{0}u(t))\diff t+\sum_{i=1}^{r-1}(A_ix(t)+B_iu(t))\diff w_{i}(t),
\]
rather than \cref{eq:scare-system:scare}.
In this scenario, instead of \cref{eq:scare}, the induced SCARE is the so-called generalized SCARE
	\begin{align}
		&Q+A_{{0}}^{\T}XE+E^{\T}XA_{{0}} + \sum_{i=1}^{{r-1}} A_i^{\T}XA_i
		\alt{\nonumber\\&}
			\ulub{\alt{\!}=\alt{\!}}{(L\alt{\!}+\alt{\!}E^{\T}\alt{\!}XB_{{0}}\alt{\!}+\alt{\!\!}\sum_{i=1}^{{r-1}}\alt{\!} A_i^{\T}\alt{\!}XB_i)}{
			(R\alt{\!}+\alt{\!\!}\sum_{i=1}^{{r-1}}\alt{\!} B_i^{\T}\alt{\!}XB_i)^{\alt{\!}-\alt{\!}1}\alt{\!}(}{*)^{\alt{\!}\T}\alt{\!}},\alt{\!}
\label{eq:generalized-scare}
	\end{align}
which leads to an (ordinary) SCARE if  $E$ is nonsingular:
\[
	\begin{multlined}[t]
		{C'}{}^{\T}C'+{A'}{}^{\T} X+XA' + {\what A'}{}^{\T}\ltimes X\ltimes \what A' \alt{\\}
		\ulub{=}{(X B+{\what A'}{}^{\T}\ltimes X \ltimes \what B)}{(I+\what B^{\T}\ltimes X \ltimes \what B)^{-1}(}{*)^{\T}}
		,
	\end{multlined}
\]
where $C'=CE^{-1},A'=AE^{-1},\what A'=\what AE^{-1}$.

\Cref{alg:an-radi-type-method-for-scares:sparse} can also be used to solve the generalized SCARE, with the known $E$ used in place of $E=I_n$ in the input.
The additional complexity comes from:
\begin{itemize}
	\item $A-\gamma E$ in Step~4: since $A$ and $E$ typically share the same sparse structure, $A-\gamma E$ costs $2\nnz(A)$ flops; 
	\item $(\sqrt{2\gamma}N^{-\alt{\!}\T}S)E$ is Step~8: it requires $l_{\trun}\mul_{n\times n}(E)$ flops;
		\item generating $\gamma$: $E^{-1}U$ has to be computed, which costs $sl_{\trun}\sol_{n\times n}(E)$ flops.
\end{itemize}

An alternative way to compute the solution to the generalized SCARE is treating $XE$ in \cref{eq:generalized-scare} as the unknown, which makes \cref{eq:generalized-scare} an (ordinary) SCARE.

\section{Experiments and discussions}\label{sec:experiments-and-discussions}
We will provide numerical results from two examples to illustrate the behavior of \cref{alg:an-radi-type-method-for-scares:sparse}.
All experiments were conducted in MATLAB 2021a under the Windows 10 Professional 64-bit operating system on a PC with a Intel Core i7-8700 processor at 3.20GHz and 64GB RAM. 

In general, the SCAREs correspond with the CAREs by omitting $\what A,\what B$.
To highlight the difference between various stochastic processes, we select two CAREs and modify them for our test.
For each example, we test six SCAREs:
\begin{enumerate}
	\item original CARE, namely $r=1$: $C^{\T}C+A^{\T}X+XA-XBB^{\T}X=0$;
	\item $r=2$ with noise scale $\texttt{ns} = 10^{-5}$/$10^{-4}$/$10^{-3}$/$10^{-2}$: $A,B,C$ are given as above, and $A_1,B_1$ are generated by the MATLAB function \texttt{ns*M.*sprand(M)} for $M=A,B$.
	\item $r=5$: $A,B,C$ are given as above, and $A_1,A_2,A_3,A_4$ are the four distinct $A_1$'s in item~2 (with $\texttt{ns} = 10^{-5}$/$10^{-4}$/$10^{-3}$/$10^{-2}$ respectively), and the same applied for $B_1,B_2,B_3,B_4$.
\end{enumerate}
In each example, we will use five types of shift strategies:
\begin{enumerate}
	\item ``hami $1$/$2$/$5$'': the residual Hamiltonian shifts with the prescribed $s=1,2,5$ respectively, and computing the shifts one time for the next many iterations (one shift per iteration, and recomputing  until all shifts are used).
	\item ``hami c $1$/$2$/$5$'': the residual Hamiltonian shifts with the prescribed $s=1,2,5$ respectively, and computing the shift in each iteration. 
	\item ``proj $1$/$2$/$5$'': the projection shifts with the prescribed $s=1,2,5$ respectively, and computing the shifts one time for the next many iterations (one shift per iteration, and computing again until all shifts are used).
	\item ``proj c $1$/$2$/$5$'': the projection shifts with the prescribed $s=1,2,5$ respectively, and computing  the shift in each iteration.
\end{enumerate}
Note that ``hami c *'' and ``proj *'' are respectively the default choices 
of the package M-M.E.S.S.\ version~2.1 \cite{SaaKB21-mmess-2.1} of the RADI for CAREs.

The stopping criteria are: $\nres_*(X_{(k)})< 10^{-12}$; or reaching $300$ iterations. 
The compression or truncation uses a cross-product-free algorithm (actually MATLAB built-in function \texttt{svd}).
To adapt with the accuracy, the truncation criterion is: $\frac{\N{\Omega_{(k)}}_*}{\N{CC^{\T}}_*}\le 3.33\times10^{-15}=10^{-12}/300$.

\newcommand\plottime[6][1]{%
	\tikz[yscale=1.4,xscale=0.8*#1]{
		\fill[black]           (0,0) rectangle (#2, .1);
		\fill[black!30!white] (#2,0) rectangle (#3, .1);
		\fill[black!60!white] (#3,0) rectangle (#4, .1);
		\fill[black!10!white] (#4,0) rectangle (#5, .1);
		\fill[black!40!white] (#5,0) rectangle (#6, .1);
}}
\begin{example}[Rail]\label{eg:rail}
	The example is a version of the steel profile cooling model from the Oberwolfach Model Reduction Benchmark Collection, hosted at MORwiki \cite{morwiki_steel}.
	The data include $A\prec 0, E\succ 0, B,C$ with $m=7,l=6$ and four different $n=1357,5177,20209,79841$, while the corresponding $\nnz(A)\alt{\!}=\alt{\!}8985, 35185, 139233, 553921$ respectively.
	This is an example on the generalized SCARE.


	Since $A\prec 0$, the system is stable at $r=1$. 
	However, the systems at $r=2$ and $5$ are still not guaranteed to be stable or stabilizable.

	\cref{tab:results-rail} collects the basic performance data, including iteration counts, the numbers of columns of $\Xi$, and the running time, of all $4\times 6\times 12$ experiments (different $n$'s, different $r$'s, different shift strategies). 
	Additional information is included  in the column ``remark'':
	if the criterion on $\nres_*$ is not met, the $\nres_*(X_{(300)})$ is reported;
	otherwise, a color bar in five shades of gray  \plottime[.3]{1}{2}{3}{4}{5} indicates the timings for 1) $\gamma$ in Step~3, 2) $C_A,F_A$ in Step~4, 3) $C_{\gamma}\ltimes A$ in Step~5, 4) truncated SVD in Step~14, and 5) others.

	\cref{fig:rail} shows the 
	convergence behavior for the case $n=79841$ in terms of both time (in seconds) vs.\ accuracy ($\nres_*$) and the numbers of columns of $\Xi$ vs.\ accuracy.
	Note that some curves were truncated to enhance the visibility of those convergent curves.
\end{example}
\begin{example}[Lung2$-$]\label{eg:lung2-}
	The example is generated in this way: $-A$ is the matrix \texttt{lung2} in the SuiteSparse Matrix Collection \cite{davisH2011university} (formerly the University of Florida Sparse Matrix Collection), modelling temperature and water vapor transport in the human lung; $B,C$ are generated by MATLAB function \texttt{rand}.
	Here $n=109460,\nnz(A)=492564,m=10,l=5$ and $B,C$ are dense matrices.


	In this example $A$ is nonsymmetric with its eigenvalues located in the left half-plane, thus $A$ is stable. 
	Similarly to \cref{eg:rail}, the system at $r=1$ is stable, while the systems at $r=2$ and $5$ are not guaranteed to be stable or stabilizable.

	An additional stopping criterion is implemented in this exmaple: the number of columns of $\Xi$ reaches $50000$, for $\Xi$ is a dense matrix and a $n\times 50000$ matrix occupies about 40.8GB. 

	Moreover, we conduct another round of test to observe the effect of truncations, with a different truncation criterion: $\frac{\N{\Omega_{(k)}}_*}{\N{CC^{\T}}_*}\le 10^{-10}$.
	To adapt with that, the stopping criteria are modified to: $\nres_*(X_{(k)})< 10^{-10}$; or reaching $300$ iterations; or the difference of $\nres$ and the accumulated truncation errors in all compressions diminishes under $10^{-10}$.

	Results from both rounds are collected in \cref{tab:results-lung2-}, where the flag ``m'' indicates that the number of columns of $\Xi$ approaches $50000$, and the flag ``t'' indicates that the total $\nres$ does not meet the criteria but does when the accumulated truncation error is excluded.
	The 
	convergent behavior for the large truncation case is shown in \cref{fig:lung2}.
\end{example}

\begin{sidewaystable}[htp]
	\centering
	\scriptsize
	\caption{Results: Rail}
	\label{tab:results-rail}
	\begin{tabular}{@{}c@{\,}*{6}{|@{\,}r@{\,}r@{\,}r@{\,}l@{\,}}}
		\hline
		Rail
		& \multicolumn{3}{l}{\makebox[0pt][l]{$r=1$}}         & 
		& \multicolumn{3}{l}{\makebox[0pt][l]{$r=2,10^{-5}$}} & 
		& \multicolumn{3}{l}{\makebox[0pt][l]{$r=2,10^{-4}$}} & 
		& \multicolumn{3}{l}{\makebox[0pt][l]{$r=2,10^{-3}$}} & 
		& \multicolumn{3}{l}{\makebox[0pt][l]{$r=2,10^{-2}$}} & 
		& \multicolumn{3}{l}{\makebox[0pt][l]{$r=5$}}         & 
\\ \hline
		shift
		& ite    & dim      & time  & remark 
		& ite    & dim      & time  & remark 
		& ite    & dim      & time  & remark 
		& ite    & dim      & time  & remark 
		& ite    & dim      & time  & remark 
		& ite    & dim      & time  & remark 
		\\
		\hline
		$n=1357$
		&&&& \plottime[4]{0}{0}{0}{0}{.1}$\,=\,$.1s
		&&&& \plottime[4]{0}{0}{0}{0}{.1}$\,=\,$.1s
		&&&& \plottime[4]{0}{0}{0}{0}{.1}$\,=\,$.1s
		&&&& \plottime[4]{0}{0}{0}{0}{.1}$\,=\,$.1s
		&&&& \plottime[.5]{0}{0}{0}{0}{.8}$\,=\,$.8s
		&&&& \plottime[.5]{0}{0}{0}{0}{.8}$\,=\,$.8s
\\ \hline
  hami 1   & 38  & 228   & 0.137  & \plottime[4]{0.021}{0.074}{0.075}{0.080}{0.137} & 38  & 228   & 0.132  & \plottime[4]{0.009}{0.061}{0.070}{0.077}{0.132} & 38  & 234   & 0.108  & \plottime[4]{0.004}{0.056}{0.062}{0.070}{0.108} & 39  & 410   & 0.131  & \plottime[4]{0.003}{0.060}{0.068}{0.082}{0.131} & 54  & 1030  & 0.240  & \plottime[.5]{0.005}{0.091}{0.105}{0.149}{0.240} & 69  & 1455  & 0.688  & \plottime[.5]{0.008}{0.125}{0.221}{0.458}{0.688} \\
  hami 2   & 39  & 212   & 0.111  & \plottime[4]{0.005}{0.059}{0.059}{0.064}{0.111} & 39  & 212   & 0.121  & \plottime[4]{0.004}{0.057}{0.064}{0.071}{0.121} & 39  & 218   & 0.110  & \plottime[4]{0.003}{0.057}{0.064}{0.071}{0.110} & 58  & 534   & 0.185  & \plottime[4]{0.005}{0.088}{0.099}{0.117}{0.185} & 97  & 1572  & 0.402  & \plottime[.5]{0.012}{0.165}{0.189}{0.256}{0.402} & 125 & 2120  & 1.051  & \plottime[.5]{0.029}{0.236}{0.381}{0.693}{1.051} \\
  hami 5   & 52  & 282   & 0.119  & \plottime[4]{0.006}{0.076}{0.076}{0.080}{0.119} & 52  & 282   & 0.147  & \plottime[4]{0.006}{0.078}{0.087}{0.096}{0.147} & 53  & 294   & 0.147  & \plottime[4]{0.006}{0.078}{0.087}{0.096}{0.147} & 112 & 878   & 0.342  & \plottime[4]{0.017}{0.175}{0.196}{0.226}{0.342} & 160 & 2279  & 0.688  & \plottime[.5]{0.032}{0.295}{0.336}{0.433}{0.688} & 227 & 3328  & 1.755  & \plottime[.5]{0.059}{0.427}{0.677}{1.157}{1.755} \\
  hami c 1 & 36  & 210   & 0.119  & \plottime[4]{0.025}{0.073}{0.074}{0.078}{0.119} & 36  & 210   & 0.114  & \plottime[4]{0.017}{0.066}{0.072}{0.078}{0.114} & 36  & 205   & 0.132  & \plottime[4]{0.026}{0.074}{0.081}{0.087}{0.132} & 40  & 422   & 0.161  & \plottime[4]{0.029}{0.087}{0.095}{0.111}{0.161} & 51  & 1052  & 0.312  & \plottime[.5]{0.078}{0.162}{0.176}{0.221}{0.312} & 59  & 1428  & 1.048  & \plottime[.5]{0.129}{0.232}{0.316}{0.815}{1.048} \\
  hami c 2 & 30  & 176   & 0.088  & \plottime[4]{0.020}{0.061}{0.061}{0.064}{0.088} & 30  & 176   & 0.102  & \plottime[4]{0.021}{0.062}{0.067}{0.072}{0.102} & 28  & 174   & 0.096  & \plottime[4]{0.020}{0.058}{0.063}{0.068}{0.096} & 33  & 380   & 0.166  & \plottime[4]{0.050}{0.099}{0.106}{0.121}{0.166} & 65  & 1026  & 0.432  & \plottime[.5]{0.184}{0.283}{0.298}{0.339}{0.432} & 55  & 1213  & 0.895  & \plottime[.5]{0.299}{0.390}{0.466}{0.688}{0.895} \\
  hami c 5 & 29  & 174   & 0.122  & \plottime[4]{0.054}{0.094}{0.094}{0.097}{0.122} & 29  & 174   & 0.132  & \plottime[4]{0.053}{0.092}{0.097}{0.103}{0.132} & 29  & 177   & 0.135  & \plottime[4]{0.054}{0.094}{0.099}{0.105}{0.135} & 31  & 354   & 0.258  & \plottime[4]{0.152}{0.197}{0.204}{0.218}{0.258} & 300 & 4906  & 4.392  & \quad4.884e-05
& 300 & 5219  & 6.259  & \quad4.858e-05
 \\
  proj 1   & 37  & 220   & 0.107  & \plottime[4]{0.012}{0.063}{0.064}{0.068}{0.107} & 37  & 220   & 0.119  & \plottime[4]{0.007}{0.057}{0.064}{0.071}{0.119} & 38  & 228   & 0.109  & \plottime[4]{0.003}{0.056}{0.063}{0.070}{0.109} & 36  & 414   & 0.127  & \plottime[4]{0.003}{0.056}{0.063}{0.080}{0.127} & 57  & 1222  & 0.271  & \plottime[.5]{0.004}{0.097}{0.113}{0.171}{0.271} & 65  & 1742  & 0.890  & \plottime[.5]{0.015}{0.130}{0.230}{0.587}{0.890} \\
  proj 2   & 35  & 210   & 0.084  & \plottime[4]{0.003}{0.050}{0.050}{0.054}{0.084} & 35  & 210   & 0.097  & \plottime[4]{0.003}{0.050}{0.056}{0.062}{0.097} & 35  & 216   & 0.097  & \plottime[4]{0.003}{0.051}{0.056}{0.063}{0.097} & 58  & 532   & 0.185  & \plottime[4]{0.004}{0.087}{0.098}{0.118}{0.185} & 94  & 1591  & 0.408  & \plottime[.5]{0.006}{0.157}{0.183}{0.259}{0.408} & 120 & 2223  & 1.170  & \plottime[.5]{0.013}{0.208}{0.357}{0.791}{1.170} \\
  proj 5   & 68  & 314   & 0.152  & \plottime[4]{0.005}{0.095}{0.095}{0.101}{0.152} & 68  & 314   & 0.187  & \plottime[4]{0.005}{0.098}{0.109}{0.119}{0.187} & 68  & 323   & 0.177  & \plottime[4]{0.005}{0.096}{0.107}{0.117}{0.177} & 112 & 887   & 0.342  & \plottime[4]{0.010}{0.169}{0.189}{0.220}{0.342} & 146 & 1960  & 0.551  & \plottime[.5]{0.014}{0.236}{0.271}{0.355}{0.551} & 210 & 2846  & 1.581  & \plottime[.5]{0.027}{0.354}{0.571}{1.076}{1.581} \\
  proj c 1 & 30  & 180   & 0.082  & \plottime[4]{0.015}{0.055}{0.055}{0.058}{0.082} & 30  & 180   & 0.095  & \plottime[4]{0.013}{0.054}{0.059}{0.064}{0.095} & 31  & 196   & 0.100  & \plottime[4]{0.014}{0.056}{0.061}{0.068}{0.100} & 31  & 453   & 0.142  & \plottime[4]{0.023}{0.070}{0.077}{0.096}{0.142} & 52  & 1947  & 0.436  & \plottime[.5]{0.091}{0.190}{0.210}{0.306}{0.436} & 58  & 3059  & 1.663  & \plottime[.5]{0.156}{0.284}{0.422}{1.206}{1.663} \\
  proj c 2 & 49  & 294   & 0.145  & \plottime[4]{0.028}{0.099}{0.099}{0.104}{0.145} & 49  & 294   & 0.161  & \plottime[4]{0.028}{0.095}{0.103}{0.112}{0.161} & 49  & 304   & 0.163  & \plottime[4]{0.028}{0.097}{0.105}{0.115}{0.163} & 49  & 1055  & 0.327  & \plottime[4]{0.097}{0.178}{0.192}{0.238}{0.327} & 268 & 10375 & 2.914  & \plottime[.5]{1.167}{1.682}{1.787}{2.264}{2.914} & 300 & 16991 & 10.857 & \quad2.786e-03
\\
  proj c 5 & 68  & 408   & 0.236  & \plottime[4]{0.083}{0.176}{0.176}{0.183}{0.236} & 68  & 408   & 0.273  & \plottime[4]{0.086}{0.182}{0.194}{0.206}{0.273} & 68  & 419   & 0.272  & \plottime[4]{0.085}{0.180}{0.192}{0.204}{0.272} & 300 & 7683  & 4.904  & \quad1.588e-02
& 300 & 11885 & 8.996  & \quad1.749e-02
& 300 & 17881 & 21.923 & \quad1.771e-02
\\
\hline
 		$n=5177$
&&&& \plottime{0}{0}{0}{0}{.4}$\,=\,$.4s
&&&& \plottime{0}{0}{0}{0}{.4}$\,=\,$.4s
&&&& \plottime[.5]{0}{0}{0}{0}{.8}$\,=\,$.8s
&&&& \plottime[.5]{0}{0}{0}{0}{.8}$\,=\,$.8s
&&&& \plottime[.25]{0}{0}{0}{0}{1.6}$\,=\,$1.6s
&&&& \plottime[.125]{0}{0}{0}{0}{3.2}$\,=\,$3.2s
\\ \hline
  hami 1   & 43  & 258   & 0.393  & \plottime{0.016}{0.254}{0.255}{0.264}{0.393} & 43  & 258   & 0.449  & \plottime{0.013}{0.261}{0.292}{0.309}{0.449} & 43  & 267   & 0.479  & \plottime[.5]{0.019}{0.274}{0.305}{0.324}{0.479} & 46  & 534   & 0.578  & \plottime[.5]{0.011}{0.306}{0.345}{0.403}{0.578} & 80  & 1868  & 1.825  & \plottime[.25]{0.020}{0.719}{0.865}{1.140}{1.825} & 90  & 2655  & 4.700  & \plottime[.125]{0.024}{0.901}{1.673}{2.912}{4.700} \\
  hami 2   & 42  & 240   & 0.388  & \plottime{0.009}{0.258}{0.258}{0.266}{0.388} & 42  & 240   & 0.461  & \plottime{0.010}{0.264}{0.296}{0.312}{0.461} & 42  & 248   & 0.450  & \plottime[.5]{0.011}{0.260}{0.291}{0.308}{0.450} & 64  & 710   & 0.840  & \plottime[.5]{0.017}{0.439}{0.499}{0.579}{0.840} & 111 & 2568  & 2.426  & \plottime[.25]{0.052}{1.008}{1.186}{1.578}{2.426} & 137 & 3738  & 6.502  & \plottime[.125]{0.043}{1.341}{2.394}{4.075}{6.502} \\
  hami 5   & 68  & 370   & 0.651  & \plottime{0.020}{0.426}{0.426}{0.440}{0.651} & 68  & 370   & 0.759  & \plottime{0.019}{0.441}{0.493}{0.519}{0.759} & 67  & 374   & 0.724  & \plottime[.5]{0.019}{0.423}{0.472}{0.498}{0.724} & 132 & 1161  & 1.552  & \plottime[.5]{0.047}{0.859}{0.966}{1.081}{1.552} & 187 & 3759  & 3.654  & \plottime[.25]{0.065}{1.567}{1.839}{2.383}{3.654} & 237 & 5448  & 9.384  & \plottime[.125]{0.096}{2.096}{3.643}{5.938}{9.384} \\
  hami c 1 & 42  & 249   & 0.480  & \plottime{0.082}{0.341}{0.341}{0.350}{0.480} & 42  & 249   & 0.531  & \plottime{0.085}{0.338}{0.370}{0.387}{0.531} & 37  & 232   & 0.481  & \plottime[.5]{0.076}{0.300}{0.329}{0.346}{0.481} & 48  & 522   & 0.711  & \plottime[.5]{0.120}{0.424}{0.466}{0.526}{0.711} & 70  & 1669  & 1.948  & \plottime[.25]{0.371}{0.970}{1.100}{1.354}{1.948} & 152 & 2983  & 5.952  & \plottime[.125]{0.681}{1.794}{2.679}{4.004}{5.952} \\
  hami c 2 & 38  & 228   & 0.447  & \plottime{0.099}{0.326}{0.326}{0.334}{0.447} & 38  & 228   & 0.505  & \plottime{0.097}{0.328}{0.356}{0.371}{0.505} & 36  & 223   & 0.470  & \plottime[.5]{0.091}{0.306}{0.333}{0.348}{0.470} & 42  & 529   & 0.715  & \plottime[.5]{0.203}{0.460}{0.498}{0.558}{0.715} & 300 & 6933  & 9.084  & \quad1.999e-05
  & 288 & 4099  & 9.427  & \plottime[.15]{1.899}{3.773}{5.112}{6.778}{9.427} \\
  hami c 5 & 34  & 191   & 0.451  & \plottime{0.181}{0.363}{0.363}{0.369}{0.451} & 34  & 191   & 0.564  & \plottime{0.194}{0.416}{0.438}{0.450}{0.564} & 34  & 199   & 0.499  & \plottime[.5]{0.189}{0.371}{0.393}{0.405}{0.499} & 115 & 1138  & 2.571  & \plottime[.5]{1.314}{1.978}{2.068}{2.182}{2.571} & 300 & 6667  & 15.229 & \quad6.051e-05
& 300 & 7369  & 22.933 & \quad5.950e-05
\\
  proj 1   & 39  & 234   & 0.373  & \plottime{0.016}{0.242}{0.242}{0.250}{0.373} & 39  & 234   & 0.469  & \plottime{0.015}{0.272}{0.303}{0.320}{0.469} & 39  & 245   & 0.437  & \plottime[.5]{0.013}{0.249}{0.279}{0.296}{0.437} & 51  & 593   & 0.708  & \plottime[.5]{0.014}{0.369}{0.418}{0.494}{0.708} & 70  & 1779  & 1.653  & \plottime[.25]{0.015}{0.649}{0.779}{1.041}{1.653} & 85  & 2586  & 4.908  & \plottime[.125]{0.023}{0.844}{1.573}{3.125}{4.908} \\
  proj 2   & 45  & 262   & 0.445  & \plottime{0.011}{0.295}{0.295}{0.304}{0.445} & 45  & 262   & 0.489  & \plottime{0.011}{0.281}{0.315}{0.332}{0.489} & 45  & 273   & 0.528  & \plottime[.5]{0.012}{0.297}{0.333}{0.353}{0.528} & 62  & 683   & 0.830  & \plottime[.5]{0.015}{0.442}{0.501}{0.580}{0.830} & 105 & 2352  & 2.205  & \plottime[.25]{0.028}{0.888}{1.054}{1.384}{2.205} & 134 & 3342  & 6.202  & \plottime[.125]{0.031}{1.167}{2.141}{3.981}{6.202} \\
  proj 5   & 67  & 334   & 0.598  & \plottime{0.016}{0.391}{0.391}{0.402}{0.598} & 67  & 334   & 0.721  & \plottime{0.026}{0.427}{0.475}{0.498}{0.721} & 68  & 350   & 0.689  & \plottime[.5]{0.017}{0.405}{0.453}{0.476}{0.689} & 141 & 1127  & 1.744  & \plottime[.5]{0.034}{0.911}{1.026}{1.157}{1.744} & 172 & 2905  & 2.951  & \plottime[.25]{0.043}{1.314}{1.536}{1.915}{2.951} & 227 & 3994  & 7.795  & \plottime[.125]{0.056}{1.772}{3.039}{5.092}{7.795} \\
  proj c 1 & 34  & 204   & 0.380  & \plottime{0.066}{0.272}{0.272}{0.279}{0.380} & 34  & 204   & 0.418  & \plottime{0.064}{0.264}{0.290}{0.304}{0.418} & 35  & 231   & 0.453  & \plottime[.5]{0.071}{0.286}{0.312}{0.328}{0.453} & 34  & 555   & 0.652  & \plottime[.5]{0.105}{0.335}{0.376}{0.451}{0.652} & 58  & 2466  & 2.513  & \plottime[.25]{0.480}{1.098}{1.267}{1.656}{2.513} & 62  & 3876  & 7.896  & \plottime[.125]{0.786}{1.596}{2.609}{5.283}{7.896} \\
  proj c 2 & 55  & 330   & 0.623  & \plottime{0.127}{0.456}{0.456}{0.468}{0.623} & 55  & 330   & 0.706  & \plottime{0.134}{0.462}{0.503}{0.525}{0.706} & 54  & 352   & 0.681  & \plottime[.5]{0.130}{0.445}{0.484}{0.509}{0.681} & 58  & 1067  & 1.438  & \plottime[.5]{0.397}{0.827}{0.908}{1.058}{1.438} & 300 & 13979 & 16.030 & \quad9.740e-04
& 300 & 21221 & 46.063 & \quad1.319e-02
\\
  proj c 5 & 77  & 462   & 1.160  & \plottime{0.499}{0.939}{0.940}{0.956}{1.160} & 77  & 462   & 1.196  & \plottime{0.456}{0.879}{0.931}{0.962}{1.196} & 145 & 980   & 2.600  & \plottime[.5]{1.062}{1.880}{1.983}{2.123}{2.600} & 300 & 9049  & 18.157 & \quad1.824e-02
& 300 & 13645 & 26.907 & \quad1.793e-02
& 300 & 21405 & 77.577 & \quad1.850e-02
\\
\hline
		$n=20209$
&&&& \plottime[.2]{0}{0}{0}{0}{2}$\,=\,$2s
&&&& \plottime[.2]{0}{0}{0}{0}{2}$\,=\,$2s
&&&& \plottime[.2]{0}{0}{0}{0}{2}$\,=\,$2s
&&&& \plottime[.1]{0}{0}{0}{0}{4}$\,=\,$4s
&&&& \plottime[.02]{0}{0}{0}{0}{20}$\,=\,$20s
&&&& \plottime[.005]{0}{0}{0}{0}{80}$\,=\,$80s
\\ \hline
  hami 1    & 44  & 264   & 2.139   & \plottime[.2]{0.055}{1.726}{1.726}{1.748}{2.139} & 44 & 264 & 2.291 & \plottime[.2]{0.050}{1.643}{1.783}{1.853}{2.291} & 44  & 277  & 2.348  & \plottime[.2]{0.049}{1.658}{1.801}{1.876}{2.348} & 58  & 702  & 4.090  & \plottime[.1]{0.052}{2.414}{2.706}{2.975}{4.090}      & 97  & 2519  & 10.101 & \plottime[.02]{0.080}{4.611}{5.410}{6.576}{10.101}     & 114 & 3829  & 27.219  & \plottime[.005]{0.102}{5.908}{10.561}{17.275}{27.219}     \\
  hami 2    & 46  & 276   & 2.119   & \plottime[.2]{0.046}{1.718}{1.718}{1.739}{2.119} & 46 & 276 & 2.550 & \plottime[.2]{0.046}{1.834}{1.983}{2.061}{2.550} & 46  & 284  & 2.543  & \plottime[.2]{0.052}{1.811}{1.964}{2.043}{2.543} & 73  & 881  & 5.085  & \plottime[.1]{0.061}{3.010}{3.364}{3.714}{5.085}      & 129 & 3315  & 13.241 & \plottime[.02]{0.105}{6.052}{7.097}{8.587}{13.241}     & 156 & 5105  & 36.127  & \plottime[.005]{0.168}{8.007}{14.335}{22.901}{36.127}     \\
  hami 5    & 72  & 407   & 3.349   & \plottime[.2]{0.077}{2.731}{2.731}{2.762}{3.349} & 72 & 407 & 3.798 & \plottime[.2]{0.077}{2.746}{2.970}{3.078}{3.798} & 72  & 412  & 3.835  & \plottime[.2]{0.079}{2.762}{2.987}{3.100}{3.835} & 159 & 1438 & 9.801  & \plottime[.1]{0.160}{6.245}{6.897}{7.407}{9.801}      & 248 & 5532  & 23.068 & \plottime[.02]{0.241}{11.172}{12.959}{15.368}{23.068}  & 272 & 8304  & 58.747  & \plottime[.005]{0.390}{13.668}{24.107}{37.349}{58.747}    \\
  hami c 1  & 59  & 354   & 3.009   & \plottime[.2]{0.377}{2.496}{2.496}{2.524}{3.009} & 59 & 354 & 3.513 & \plottime[.2]{0.410}{2.590}{2.790}{2.890}{3.513} & 39  & 247  & 2.333  & \plottime[.2]{0.275}{1.698}{1.828}{1.899}{2.333} & 52  & 634  & 4.139  & \plottime[.1]{0.581}{2.625}{2.881}{3.120}{4.139}      & 73  & 1830  & 8.720  & \plottime[.02]{1.452}{4.725}{5.310}{6.135}{8.720}      & 95  & 2855  & 22.777  & \plottime[.005]{2.268}{6.761}{10.360}{15.316}{22.777}     \\
  hami c 2  & 42  & 252   & 2.338   & \plottime[.2]{0.447}{1.953}{1.953}{1.973}{2.338} & 42 & 252 & 2.617 & \plottime[.2]{0.450}{1.968}{2.104}{2.173}{2.617} & 43  & 250  & 2.672  & \plottime[.2]{0.452}{2.014}{2.152}{2.221}{2.672} & 55  & 687  & 4.952  & \plottime[.1]{1.162}{3.298}{3.580}{3.851}{4.952}     & 300 & 8956  & 46.279  & \quad2.514e-04
& 300 & 10817 & 90.851  & \quad2.916e-04
\\
  hami c 5  & 40  & 225   & 2.697   & \plottime[.2]{0.908}{2.339}{2.339}{2.357}{2.697} & 40 & 225 & 2.867 & \plottime[.2]{0.878}{2.279}{2.404}{2.463}{2.867} & 41  & 239  & 2.977  & \plottime[.2]{0.932}{2.371}{2.499}{2.564}{2.977} & 143 & 1410 & 13.949 & \plottime[.1]{5.335}{10.579}{11.182}{11.682}{13.949} & 300 & 8494  & 75.660  & \quad1.165e-04
& 300 & 10370 & 133.669 & \quad1.211e-04
\\
  proj 1    & 48  & 277   & 2.367   & \plottime[.2]{0.060}{1.910}{1.910}{1.934}{2.367} & 48 & 277 & 2.486 & \plottime[.2]{0.055}{1.791}{1.941}{2.014}{2.486} & 48  & 290  & 2.546  & \plottime[.2]{0.054}{1.805}{1.961}{2.040}{2.546} & 58  & 757  & 4.240  & \plottime[.1]{0.053}{2.453}{2.754}{3.059}{4.240}      & 74  & 2214  & 8.594  & \plottime[.02]{0.060}{3.647}{4.332}{5.491}{8.594}      & 90  & 3388  & 24.858  & \plottime[.005]{0.077}{4.874}{8.988}{15.992}{24.858}      \\
  proj 2    & 46  & 276   & 2.200   & \plottime[.2]{0.046}{1.783}{1.783}{1.805}{2.200} & 46 & 276 & 2.521 & \plottime[.2]{0.049}{1.811}{1.962}{2.039}{2.521} & 46  & 289  & 2.477  & \plottime[.2]{0.049}{1.752}{1.902}{1.983}{2.477} & 65  & 830  & 4.714  & \plottime[.1]{0.073}{2.746}{3.081}{3.412}{4.714}      & 114 & 3091  & 12.208 & \plottime[.02]{0.090}{5.460}{6.435}{7.937}{12.208}     & 143 & 4547  & 33.706  & \plottime[.005]{0.144}{7.259}{12.961}{21.845}{33.706}     \\
  proj 5    & 69  & 411   & 3.356   & \plottime[.2]{0.080}{2.733}{2.733}{2.766}{3.356} & 69 & 411 & 3.725 & \plottime[.2]{0.077}{2.672}{2.897}{3.009}{3.725} & 69  & 412  & 3.797  & \plottime[.2]{0.076}{2.707}{2.936}{3.052}{3.797} & 143 & 1345 & 8.785  & \plottime[.1]{0.138}{5.620}{6.208}{6.679}{8.785}      & 197 & 4121  & 17.840 & \plottime[.02]{0.166}{8.727}{10.092}{12.046}{17.840}   & 260 & 5929  & 46.365  & \plottime[.005]{0.304}{12.078}{20.196}{30.744}{46.365}    \\
  proj c 1  & 38  & 228   & 1.958   & \plottime[.2]{0.237}{1.611}{1.611}{1.629}{1.958} & 38 & 228 & 2.210 & \plottime[.2]{0.242}{1.628}{1.750}{1.812}{2.210} & 38  & 269  & 2.380  & \plottime[.2]{0.285}{1.681}{1.820}{1.902}{2.380} & 38  & 689  & 3.741  & \plottime[.1]{0.541}{2.158}{2.402}{2.710}{3.741}      & 51  & 2474  & 10.286 & \plottime[.02]{1.801}{4.728}{5.438}{6.851}{10.286}     & 53  & 4000  & 32.607  & \plottime[.005]{3.062}{6.932}{11.232}{22.008}{32.607}     \\
  proj c 2  & 61  & 366   & 3.467   & \plottime[.2]{0.665}{2.908}{2.908}{2.937}{3.467} & 61 & 366 & 3.817 & \plottime[.2]{0.658}{2.875}{3.075}{3.177}{3.817} & 60  & 424  & 4.042  & \plottime[.2]{0.727}{2.941}{3.160}{3.292}{4.042} & 68  & 1583 & 8.919  & \plottime[.1]{2.361}{5.353}{5.872}{6.612}{8.919}      & 273 & 13870 & 69.869 & \plottime[.02]{23.215}{39.066}{42.952}{51.177}{69.869} & 297 & 22034 & 199.092 & \plottime[.005]{42.077}{63.506}{87.284}{141.668}{199.092} \\
  proj c 5  & 85  & 510   & 5.781   & \plottime[.2]{2.032}{5.040}{5.040}{5.078}{5.781} & 85 & 510 & 6.323 & \plottime[.2]{2.047}{5.059}{5.328}{5.469}{6.323} & 300 & 2413 & 25.151 & \quad1.785e-02
            & 300 & 10719 & 103.627 & \quad1.820e-02
           & 300 & 15399 & 167.329 & \quad1.817e-02
& 300 & 24667 & 469.599 & \quad1.834e-02
\\
\hline
		$n=79841$
&&&& \plottime[.04]{0}{0}{0}{0}{10}$\,=\,$10s
&&&& \plottime[.04]{0}{0}{0}{0}{10}$\,=\,$10s
&&&& \plottime[.04]{0}{0}{0}{0}{10}$\,=\,$10s
&&&& \plottime[.02]{0}{0}{0}{0}{20}$\,=\,$20s
&&&& \plottime[.01]{0}{0}{0}{0}{40}$\,=\,$40s
&&&& \plottime[.005]{0}{0}{0}{0}{80}$\,=\,$80s
\\ \hline
  hami 1   & 45 & 270 & 9.069  & \plottime[.04]{0.215}{7.096}{7.097}{7.203}{9.069}     & 45 & 270 & 10.312 & \plottime[.04]{0.208}{6.990}{7.685}{7.978}{10.312}     & 45 & 290 & 10.478 & \plottime[.04]{0.203}{7.034}{7.743}{8.077}{10.478}    & 66  & 868   & 19.661   & \plottime[.02]{0.219}{11.121}{12.560}{14.036}{19.661}   & 115 & 3133  & 55.269   & \plottime[.01]{0.416}{23.847}{28.013}{35.649}{55.269}    & 139 & 4678  & 151.880  & \plottime[.005]{0.539}{30.579}{54.499}{101.067}{151.880}   \\
  hami 2   & 51 & 306 & 10.030 & \plottime[.04]{0.169}{7.843}{7.843}{7.963}{10.030}    & 51 & 306 & 11.613 & \plottime[.04]{0.168}{7.896}{8.683}{9.012}{11.613}     & 51 & 318 & 11.697 & \plottime[.04]{0.169}{7.894}{8.691}{9.045}{11.697}    & 87  & 1043  & 25.054   & \plottime[.02]{0.293}{14.408}{16.189}{18.159}{25.054}   & 173 & 4462  & 82.862   & \plottime[.01]{0.697}{36.331}{42.553}{53.852}{82.862}    & 175 & 6275  & 199.616  & \plottime[.005]{0.737}{39.296}{70.786}{131.606}{199.616}   \\
  hami 5   & 77 & 437 & 15.031 & \plottime[.04]{0.335}{11.860}{11.860}{12.032}{15.031} & 77 & 437 & 17.189 & \plottime[.04]{0.335}{11.807}{12.960}{13.423}{17.189}  & 79 & 452 & 17.680 & \plottime[.04]{0.335}{12.158}{13.339}{13.827}{17.680} & 179 & 1774  & 47.518   & \plottime[.02]{0.804}{29.042}{32.347}{35.329}{47.518}    & 300 & 7165  & 135.768  & \quad1.515e-08
& 300 & 10325 & 326.825  & \quad4.601e-09
\\
  hami c 1 & 54 & 324 & 11.932 & \plottime[.04]{1.439}{9.598}{9.598}{9.726}{11.932}    & 54 & 324 & 13.439 & \plottime[.04]{1.445}{9.489}{10.323}{10.669}{13.439}   & 49 & 294 & 12.185 & \plottime[.04]{1.308}{8.596}{9.346}{9.689}{12.185}    & 68  & 758   & 21.371   & \plottime[.02]{2.625}{13.545}{14.892}{16.242}{21.371}   & 109 & 2454  & 55.983   & \plottime[.01]{9.385}{30.595}{34.129}{40.022}{55.983}    & 86  & 2790  & 100.886  & \plottime[.005]{10.217}{28.332}{42.684}{70.486}{100.886}   \\
  hami c 2 & 47 & 266 & 10.674 & \plottime[.04]{1.743}{8.714}{8.714}{8.820}{10.674}    & 47 & 266 & 12.066 & \plottime[.04]{1.748}{8.744}{9.453}{9.737}{12.066}     & 44 & 266 & 11.513 & \plottime[.04]{1.703}{8.268}{8.945}{9.243}{11.513}    & 76  & 897   & 27.352   & \plottime[.02]{6.019}{18.227}{19.773}{21.405}{27.352}   & 300 & 10168 & 253.725  & \quad4.862e-04
& 300 & 12198 & 470.670  & \quad5.338e-04
 \\
  hami c 5 & 45 & 259 & 12.960 & \plottime[.04]{4.423}{11.081}{11.082}{11.183}{12.960} & 45 & 259 & 14.263 & \plottime[.04]{4.428}{11.095}{11.770}{12.045}{14.263}  & 45 & 268 & 14.415 & \plottime[.04]{4.487}{11.136}{11.827}{12.123}{14.415} & 159 & 1554  & 66.590   & \plottime[.02]{25.875}{50.573}{53.486}{55.918}{66.590}  & 300 & 9803  & 422.460  & \quad2.524e-04
& 300 & 12080 & 738.199  & \quad2.592e-04
 \\
  proj 1   & 53 & 308 & 10.427 & \plottime[.04]{0.246}{8.178}{8.178}{8.300}{10.427}    & 53 & 308 & 11.997 & \plottime[.04]{0.242}{8.216}{9.018}{9.342}{11.997}     & 53 & 330 & 12.222 & \plottime[.04]{0.242}{8.279}{9.098}{9.475}{12.222}    & 63  & 924   & 20.076   & \plottime[.02]{0.215}{10.862}{12.332}{14.157}{20.076}   & 89  & 2863  & 50.626   & \plottime[.01]{0.347}{20.020}{23.828}{32.410}{50.626}    & 104 & 4397  & 141.570  & \plottime[.005]{0.388}{24.788}{46.083}{93.861}{141.570}    \\
  proj 2   & 56 & 328 & 10.973 & \plottime[.04]{0.199}{8.610}{8.610}{8.740}{10.973}    & 56 & 328 & 12.625 & \plottime[.04]{0.208}{8.592}{9.453}{9.804}{12.625}     & 56 & 348 & 12.833 & \plottime[.04]{0.198}{8.658}{9.533}{9.925}{12.833}    & 93  & 1212  & 27.917   & \plottime[.02]{0.337}{15.680}{17.696}{20.002}{27.917}   & 128 & 3912  & 69.601   & \plottime[.01]{0.481}{28.337}{33.590}{44.779}{69.601}    & 155 & 5808  & 186.726  & \plottime[.005]{0.616}{35.328}{64.306}{123.652}{186.726}   \\
  proj 5   & 73 & 437 & 14.380 & \plottime[.04]{0.328}{11.312}{11.313}{11.483}{14.380} & 73 & 437 & 16.511 & \plottime[.04]{0.331}{11.262}{12.378}{12.839}{16.511}  & 73 & 459 & 16.755 & \plottime[.04]{0.329}{11.330}{12.470}{12.983}{16.755} & 140 & 1521  & 38.780   & \plottime[.02]{0.746}{23.034}{25.773}{28.468}{38.780}   & 213 & 5242  & 93.561   & \plottime[.01]{0.949}{42.120}{49.290}{61.785}{93.561}    & 273 & 7793  & 256.540  & \plottime[.005]{1.406}{57.047}{98.899}{171.274}{256.540}   \\
  proj c 1 & 41 & 246 & 9.071  & \plottime[.04]{1.072}{7.289}{7.289}{7.388}{9.071}     & 41 & 246 & 10.255 & \plottime[.04]{1.072}{7.247}{7.882}{8.144}{10.255}     & 41 & 299 & 10.809 & \plottime[.04]{1.176}{7.386}{8.064}{8.462}{10.809}    & 42  & 885   & 19.401   & \plottime[.02]{3.136}{10.679}{11.936}{13.862}{19.401}   & 63  & 3490  & 68.597   & \plottime[.01]{13.702}{30.919}{35.158}{47.296}{68.597}   & 62  & 5346  & 195.384  & \plottime[.005]{22.213}{42.780}{65.530}{137.830}{195.384}  \\
  proj c 2 & 68 & 408 & 15.509 & \plottime[.04]{2.579}{12.616}{12.616}{12.778}{15.509} & 68 & 408 & 17.506 & \plottime[.04]{2.574}{12.608}{13.648}{14.080}{17.506}  & 68 & 487 & 18.854 & \plottime[.04]{3.008}{13.281}{14.398}{15.025}{18.854} & 79  & 2040  & 48.893   & \plottime[.02]{14.430}{29.479}{32.190}{36.630}{48.893}   & 300 & 17743 & 442.668  & \quad2.876e-04
& 300 & 28129 & 1198.404 & \quad2.690e-03
\\
  proj c 5 & 93 & 558 & 27.262 & \plottime[.04]{9.628}{23.359}{23.359}{23.579}{27.262} & 93 & 558 & 30.055 & \plottime[.04]{9.587}{23.406}{24.822}{25.409}{30.055} & 300 & 2419  & 108.313  & \quad1.812e-02
& 300 & 12472 & 571.087  & \quad1.827e-02
& 300 & 17147 & 937.535  & \quad1.824e-02
& 300 & 28759 & 2664.699 & \quad1.836e-02
\\
\hline
	\end{tabular}
\end{sidewaystable}
\begin{sidewaystable}[htp]
	\centering
	\scriptsize
	\caption{Results: Lung2$-$}
	\label{tab:results-lung2-}
	\begin{tabular}{@{}c@{\,}*{6}{|@{\,}r@{\,}r@{\,}r@{\,}l@{\,}}}
		\hline
		Lung2$-$
		& \multicolumn{3}{l}{\makebox[0pt][l]{$r=1$}}         & 
		& \multicolumn{3}{l}{\makebox[0pt][l]{$r=2,10^{-5}$}} & 
		& \multicolumn{3}{l}{\makebox[0pt][l]{$r=2,10^{-4}$}} & 
		& \multicolumn{3}{l}{\makebox[0pt][l]{$r=2,10^{-3}$}} & 
		& \multicolumn{3}{l}{\makebox[0pt][l]{$r=2,10^{-2}$}} & 
		& \multicolumn{3}{l}{\makebox[0pt][l]{$r=5$}}         & 
\\ \hline
		shift
		& ite    & dim      & time  & remark 
		& ite    & dim      & time  & remark 
		& ite    & dim      & time  & remark 
		& ite    & dim      & time  & remark 
		& ite    & dim      & time  & remark 
		& ite    & dim      & time  & remark 
		\\
		\hline
&&&& \plottime[.01143]{0}{0}{0}{0}{28}$\,=\,$28s
&&&& \plottime[.00889]{0}{0}{0}{0}{36}$\,=\,$36s
&&&& \plottime[.005]{0}{0}{0}{0}{64}$\,=\,$64s
&&&& \plottime[.002]{0}{0}{0}{0}{160}$\,=\,$160s
&&&& 
&&&& 
\\\hline
hami 1 & 225 & 1068 & 72.118 &\plottime[.01143]{2.131}{51.448}{51.449}{51.998}{72.118}
& 227 & 2640 & 111.937 &\plottime[.00889]{1.839}{58.301}{63.313}{71.302}{111.937}
& 242 & 7538 & 211.564 &\plottime[.005]{1.797}{79.781}{90.810}{125.680}{211.564}
& 285 & 22685 & 562.064 &\plottime[.002]{2.170}{145.672}{175.341}{333.040}{562.064}
& 151 & 49956 & 1189.453 &\;\:\makebox[0pt]{m}\;\:3.483e-01
& 136 & 49733 & 3616.463 &\;\:\makebox[0pt]{m}\;\:6.007e-01
\\ hami 2 & 221 & 1029 & 69.421 &\plottime[.01143]{1.430}{49.513}{49.514}{50.033}{69.421}
& 238 & 2864 & 118.630 &\plottime[.00889]{1.468}{61.127}{66.458}{75.140}{118.630}
& 253 & 8231 & 217.259 &\plottime[.005]{1.778}{81.152}{93.002}{126.697}{217.259}
& 300 & 24173 & 593.969 &\quad6.230e-04
& 149 & 49745 & 1182.042 &\;\:\makebox[0pt]{m}\;\:3.622e-01
& 143 & 49958 & 3487.781 &\;\:\makebox[0pt]{m}\;\:6.280e-01
\\ hami 5 & 259 & 1202 & 81.357 &\plottime[.01143]{1.450}{58.013}{58.014}{58.629}{81.357}
& 249 & 2889 & 122.105 &\plottime[.00889]{1.518}{63.574}{69.036}{77.535}{122.105}
& 300 & 9514 & 254.030 &\quad7.083e-08
& 300 & 32476 & 714.833 &\quad7.591e-04
& 149 & 49745 & 1315.983 &\;\:\makebox[0pt]{m}\;\:3.622e-01
& 143 & 49958 & 3515.085 &\;\:\makebox[0pt]{m}\;\:6.280e-01
\\ hami c 1 & 209 & 972 & 75.662 &\plottime[.01143]{9.324}{56.011}{56.011}{56.535}{75.662}
& 215 & 2567 & 122.405 &\plottime[.00889]{17.105}{70.513}{75.326}{83.173}{122.405}
& 227 & 6833 & 232.828 &\plottime[.005]{41.455}{113.009}{123.071}{154.485}{232.828}
& 300 & 19915 & 649.161 &\quad6.253e-04
& 154 & 49816 & 1736.838 &\;\:\makebox[0pt]{m}\;\:3.437e-01
& 145 & 49714 & 4351.744 &\;\:\makebox[0pt]{m}\;\:3.628e-01
\\ hami c 2 & 214 & 1070 & 83.045 &\plottime[.01143]{13.534}{62.164}{62.165}{62.777}{83.045}
& 300 & 3111 & 172.548 &\quad1.363e-09
& 300 & 12383 & 489.086 &\quad5.200e-07
& 300 & 37584 & 1523.890 &\quad8.527e-04
& 164 & 49972 & 2988.868 &\;\:\makebox[0pt]{m}\;\:4.268e-01
& 139 & 49980 & 5576.575 &\;\:\makebox[0pt]{m}\;\:4.306e-01
\\ hami c 5 & 300 & 1500 & 142.640 &\quad8.270e-08
& 300 & 4326 & 306.000 &\quad4.553e-10
& 300 & 18448 & 1395.591 &\quad2.193e-06
& 300 & 31149 & 3209.602 &\quad9.156e-04
& 173 & 49904 & 13150.396 &\;\:\makebox[0pt]{m}\;\:4.275e-01
& 145 & 49714 & 17241.920 &\;\:\makebox[0pt]{m}\;\:4.311e-01
\\ proj 1 & 205 & 961 & 65.037 &\plottime[.01143]{1.906}{46.521}{46.521}{47.004}{65.037}
& 219 & 2416 & 105.445 &\plottime[.00889]{1.748}{55.666}{60.378}{67.377}{105.445}
& 261 & 8230 & 217.629 &\plottime[.005]{1.631}{82.528}{94.342}{126.827}{217.629}
& 300 & 27798 & 668.425 &\quad7.814e-01
& 207 & 49880 & 1296.299 &\;\:\makebox[0pt]{m}\;\:7.990e-01
& 186 & 49634 & 3254.690 &\;\:\makebox[0pt]{m}\;\:7.998e-01
\\ proj 2 & 213 & 1008 & 67.619 &\plottime[.01143]{1.326}{48.173}{48.174}{48.698}{67.619}
& 247 & 2811 & 119.852 &\plottime[.00889]{1.453}{62.522}{67.964}{76.138}{119.852}
& 300 & 9812 & 255.872 &\quad6.815e-10
& 300 & 29246 & 647.732 &\quad7.972e-01
& 248 & 49838 & 1177.181 &\;\:\makebox[0pt]{m}\;\:7.996e-01
& 216 & 49868 & 3142.694 &\;\:\makebox[0pt]{m}\;\:8.000e-01
\\ proj 5 & 279 & 1346 & 92.933 &\plottime[.01143]{1.694}{66.105}{66.106}{66.864}{92.933}
& 300 & 3726 & 151.791 &\quad6.997e-01
& 300 & 13008 & 316.733 &\quad7.979e-01
& 300 & 22186 & 498.987 &\quad7.993e-01
& 300 & 42067 & 1012.625 &\quad8.001e-01
& 223 & 49946 & 3231.155 &\;\:\makebox[0pt]{m}\;\:8.000e-01
\\ proj c 1 & 275 & 1341 & 96.837 &\plottime[.01143]{12.132}{71.781}{71.781}{72.466}{96.837}
& 300 & 3023 & 156.986 &\quad1.038e-01
& 300 & 10245 & 315.674 &\quad7.981e-01
& 300 & 9751 & 315.085 &\quad8.298e-01
& 300 & 19126 & 543.366 &\quad8.509e-01
& 300 & 31839 & 1798.403 &\quad8.613e-01
\\ proj c 2 & 300 & 1500 & 110.897 &\quad7.907e-01
& 300 & 3741 & 190.756 &\quad7.918e-01
& 300 & 4352 & 205.108 &\quad8.178e-01
& 300 & 7503 & 287.975 &\quad9.355e-01
& 300 & 10788 & 389.053 &\quad9.819e-01
& 300 & 16904 & 1025.207 &\quad9.819e-01
\\ proj c 5 & 300 & 1500 & 143.600 &\quad9.996e-01
& 300 & 1500 & 153.648 &\quad9.996e-01
& 300 & 2406 & 183.203 &\quad9.996e-01
& 300 & 4779 & 297.800 &\quad9.996e-01
& 300 & 6260 & 378.961 &\quad9.996e-01
& 300 & 10161 & 882.586 &\quad9.996e-01
\\\hline
\multicolumn{4}{l}{large truncation 1e-10}&
&&&& 
&&&& 
&&&& 
&&&& 
&&&& 
\\\hline
hami 1 & 161 & 651 & 52.061 &\;\:\makebox[0pt]{t}\;\:4.746e-10
& 166 & 667 & 60.908 &\;\:\makebox[0pt]{t}\;\:5.538e-10
& 182 & 767 & 64.277 &\;\:\makebox[0pt]{t}\;\:1.195e-09
& 152 & 1564 & 71.166 &\;\:\makebox[0pt]{t}\;\:5.306e-09
& 243 & 9503 & 258.963 &\;\:\makebox[0pt]{t}\;\:1.404e-08
& 248 & 11611 & 614.994 &\;\:\makebox[0pt]{t}\;\:1.586e-08
\\ hami 2 & 158 & 636 & 54.968 &\;\:\makebox[0pt]{t}\;\:4.785e-10
& 156 & 649 & 56.980 &\;\:\makebox[0pt]{t}\;\:5.383e-10
& 158 & 707 & 57.822 &\;\:\makebox[0pt]{t}\;\:1.377e-09
& 188 & 1823 & 85.610 &\;\:\makebox[0pt]{t}\;\:5.871e-09
& 300 & 9825 & 262.684 &\quad1.607e-08
& 300 & 11758 & 633.241 &\quad1.808e-08
\\ hami 5 & 192 & 772 & 65.792 &\;\:\makebox[0pt]{t}\;\:4.751e-10
& 189 & 787 & 67.087 &\;\:\makebox[0pt]{t}\;\:5.483e-10
& 193 & 836 & 69.512 &\;\:\makebox[0pt]{t}\;\:1.159e-09
& 199 & 1853 & 89.338 &\;\:\makebox[0pt]{t}\;\:5.601e-09
& 300 & 10346 & 271.767 &\quad2.774e-01
& 300 & 11897 & 638.888 &\quad1.842e-02
\\ hami c 1 & 148 & 615 & 52.813 &\;\:\makebox[0pt]{t}\;\:4.702e-10
& 160 & 655 & 63.621 &\;\:\makebox[0pt]{t}\;\:5.708e-10
& 300 & 1799 & 128.069 &\quad1.891e-01
& 167 & 1686 & 87.417 &\;\:\makebox[0pt]{t}\;\:6.249e-09
& 300 & 9303 & 329.797 &\quad2.080e-01
& 300 & 11466 & 701.564 &\quad2.143e-01
\\ hami c 2 & 161 & 661 & 63.173 &\;\:\makebox[0pt]{t}\;\:4.766e-10
& 165 & 669 & 67.543 &\;\:\makebox[0pt]{t}\;\:5.802e-10
& 182 & 921 & 78.344 &\;\:\makebox[0pt]{t}\;\:1.056e-09
& 198 & 1946 & 111.595 &\;\:\makebox[0pt]{t}\;\:7.888e-09
& 300 & 15307 & 567.969 &\quad3.078e-01
& 300 & 18077 & 1163.186 &\quad3.070e-01
\\ hami c 5 & 300 & 1218 & 132.194 &\quad1.209e-08
& 300 & 1237 & 134.307 &\quad3.739e-09
& 300 & 1505 & 151.814 &\quad5.314e-07
& 300 & 3610 & 255.586 &\quad7.465e-04
& 300 & 15833 & 1130.226 &\quad4.236e-01
& 300 & 18490 & 1867.696 &\quad4.280e-01
\\ proj 1 & 179 & 723 & 58.178 &\;\:\makebox[0pt]{t}\;\:3.842e-10
& 179 & 723 & 65.857 &\;\:\makebox[0pt]{t}\;\:4.182e-10
& 177 & 728 & 64.344 &\;\:\makebox[0pt]{t}\;\:9.819e-10
& 186 & 1513 & 80.139 &\;\:\makebox[0pt]{t}\;\:5.044e-09
& 279 & 9464 & 263.219 &\;\:\makebox[0pt]{t}\;\:1.702e-08
& 285 & 11940 & 627.683 &\;\:\makebox[0pt]{t}\;\:1.926e-08
\\ proj 2 & 152 & 691 & 53.414 &\;\:\makebox[0pt]{t}\;\:3.866e-10
& 154 & 709 & 57.612 &\;\:\makebox[0pt]{t}\;\:3.918e-10
& 154 & 711 & 56.334 &\;\:\makebox[0pt]{t}\;\:1.096e-09
& 179 & 1666 & 80.215 &\;\:\makebox[0pt]{t}\;\:6.465e-09
& 300 & 11228 & 285.698 &\quad7.191e-01
& 300 & 13812 & 708.079 &\quad7.860e-01
\\ proj 5 & 230 & 1075 & 75.402 &\;\:\makebox[0pt]{t}\;\:3.893e-10
& 230 & 1075 & 82.369 &\;\:\makebox[0pt]{t}\;\:4.003e-10
& 235 & 1086 & 86.206 &\;\:\makebox[0pt]{t}\;\:1.045e-09
& 300 & 2431 & 128.420 &\quad8.261e-09
& 300 & 14644 & 359.886 &\quad7.993e-01
& 300 & 15840 & 782.248 &\quad7.990e-01
\\ proj c 1 & 166 & 796 & 60.837 &\;\:\makebox[0pt]{t}\;\:3.824e-10
& 166 & 796 & 68.721 &\;\:\makebox[0pt]{t}\;\:4.204e-10
& 168 & 806 & 68.774 &\;\:\makebox[0pt]{t}\;\:9.047e-10
& 181 & 1579 & 89.043 &\;\:\makebox[0pt]{t}\;\:5.933e-09
& 300 & 11521 & 355.818 &\quad7.975e-01
& 300 & 13687 & 771.156 &\quad7.982e-01
\\ proj c 2 & 300 & 941 & 111.483 &\quad7.759e-01
& 300 & 941 & 116.237 &\quad7.778e-01
& 300 & 945 & 114.284 &\quad7.899e-01
& 300 & 2450 & 155.020 &\quad7.844e-01
& 300 & 3513 & 177.289 &\quad8.095e-01
& 300 & 3935 & 279.611 &\quad8.116e-01
\\ proj c 5 & 300 & 1500 & 141.727 &\quad9.996e-01
& 300 & 1500 & 150.211 &\quad9.996e-01
& 300 & 1500 & 153.267 &\quad9.996e-01
& 300 & 1500 & 153.331 &\quad9.996e-01
& 300 & 1802 & 163.526 &\quad9.996e-01
& 300 & 1802 & 224.638 &\quad9.996e-01
\\ \hline
	\end{tabular}
\end{sidewaystable}

\begin{figure}[t]
	\centering
	\aptfigure{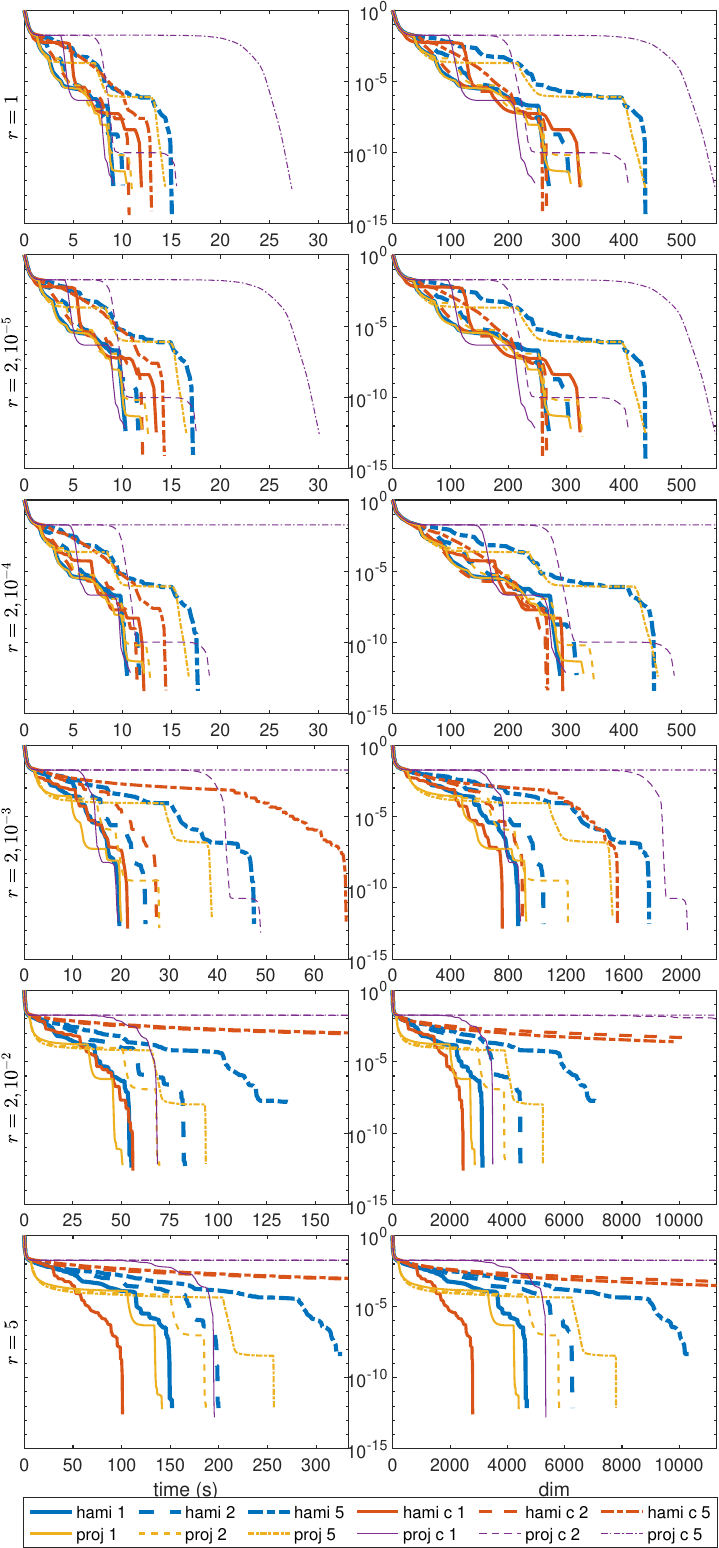}
	\caption{Rail at $n=79841$}
	\label{fig:rail}
\end{figure}
\begin{figure}[t]
	\centering
	\aptfigure{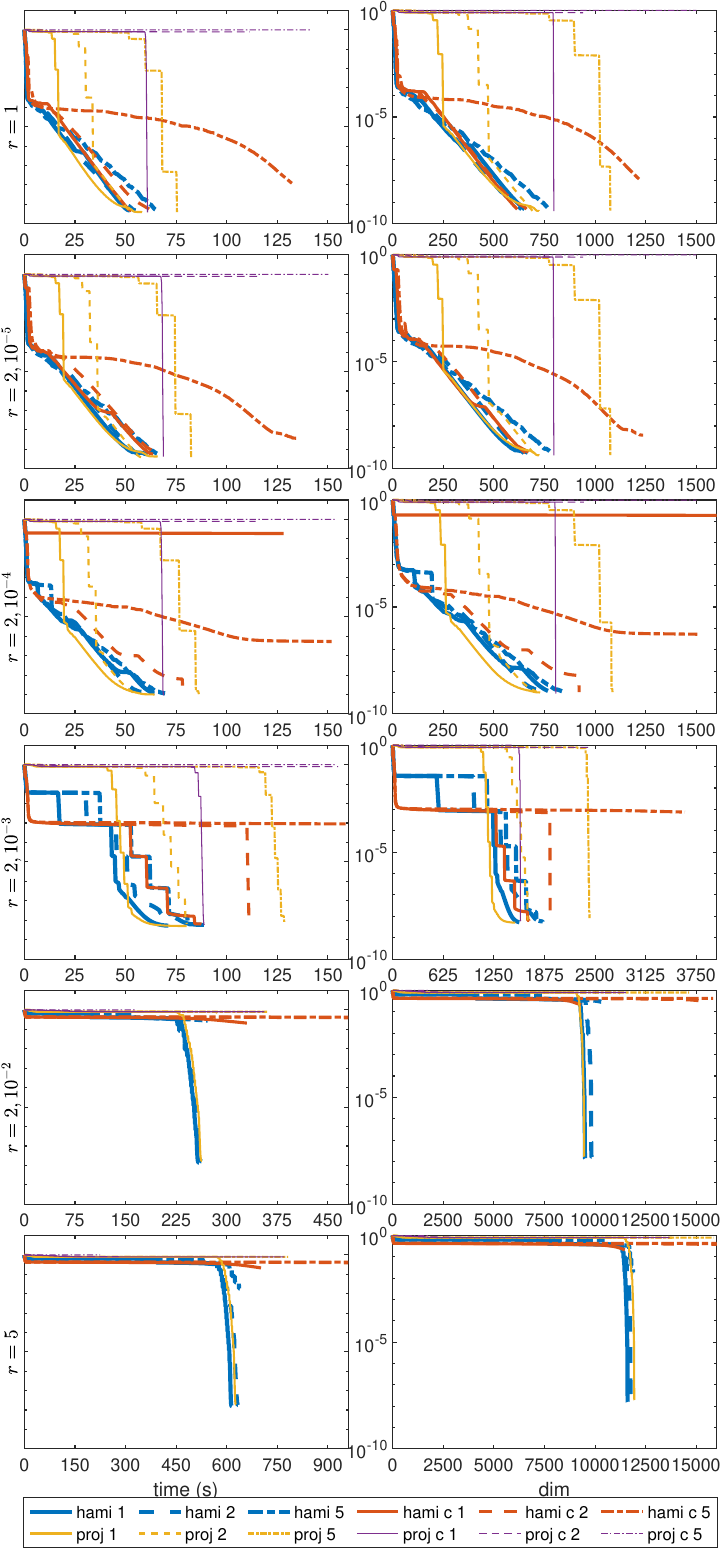}
	\caption{Lung2-: large truncation case}
	\label{fig:lung2}
\end{figure}

From the tables and figures, we can observe several notable features.

\subsection{Shift strategy}\label{para:shift-strategy}
In each shift strategy with different prescribed $s$, $s=1$ is usually the best choice because it requires the least number of iterations, generates the smallest $\Xi$, and consumes the least time with only a few exceptions.
Especially in the cases that the stochastic part is large (
$r=5$), there is no exceptions.
The reason why a projection onto a high-dimensional subspace fails to reduce the number of iterations 
is that the shifts are obtained by computing the terms on the corresponding CARE rather than the SCARE,
which would weaken the approximation for a large-dimensional subspace.

For methods that do not compute shifts frequently, including ``hami *'' and ``proj *'',
the cost for calculating the shifts is relatively tiny;
in contrast, for methods that compute shifts in each iteration, including ``hami c *'' and ``proj c *'',
the cost for calculating the shifts becomes significant and even dominate at the large stochastic part cases.

Aside from the shift cost, for 
cases with small stochastic part, 
the main cost is calculating $C_\gamma$;
for the large stochastic part cases, the primary cost is the compression.
This coincides with the analysis in \cref{ssec:complexity-list}.
The noticeable cost in the ``others'' category arises from the deep copying of data in the memory, which is unpredictable.
This drawback could be overcome by implementing the algorithm in C or a similar programming language. 


The shift strategy ``hami 1'' typically yields the best performance or is among the top few strategies.
Such result does not align precisely with the RADI method for CAREs (that favors ``hami c 1''), possibly due to the shifts are obtained by computing the terms on the corresponding CARE rather than the SCARE.

\subsection{Compression/Truncation}\label{para:truncation}
The number of retained rows during truncation, namely $l_{\trun}$, will conclusively determine the running time.
Consequently, a looser truncation criterion can significantly speed up the process but at the cost of accuracy.
For certain problematic scenarios, using a looser criterion is justifiable, as is illustrated in \cref{tab:results-lung2-} and \cref{fig:lung2}.

Comparing the results at $l=5$ in \cref{tab:results-lung2-} with those generated by RADI at $l=10$ in \cite[Table~3]{bennerBKS2018radi} on CARE ($r=1$), we observe that the compression would slow convergence. 
This is reasonable since some useful information is discarded in the compression, and it is a trade-off between convergence and storage.
Compression is necessary for the true SCARE case ($r>1$) to avoid an exponential increase of the size of $\Xi$.

\subsection{Comparison}\label{para:comparison}
In the experiment we did not compare our method (\cref{alg:an-radi-type-method-for-scares:sparse}) with some existing methods.
Here, we explain the rationale:
\begin{itemize}
	\item The comparison theorem based method \cite{freilingH2003properties} and LMI's method \cite{ramiZ2000linear} are not valid for large-scale problems.
	\item The homotopy method  \cite{zhangFCW2015homotopy}, and the Newton-type methods, including the basic variant \cite{dammH2001newtons,damm2004rational}, the modified variant \cite{guo2001iterative,ivanov2007iterations} and the GLE-driver variant \cite{fanWC2016smith},
		all face an inconspicuous but subtle and time-consuming challenge: validating the stopping criteria.
		Without a low-rank expression of the residual $\op C(X)$, one has to calculate it out explicitly and estimate its norm, resulting in a task over $n\times n$ matrices.
		In contrast, equipped with its low-rank expression \cref{eq:residual-total} and the discussion in \cref{ssec:storage-reduction}, it is reduced to a task over $rl_{\trun}\times rl_{\trun}$ matrices in our method.

	\item For  Newton-type methods, a critical challenge is finding a stabilizing initial approximation. \cite{fanWC2016smith} highlights this issue three times
, arguing that initial stabilization for Newton's method is a difficult open problem, and no resolution was provided during numerical tests.
		To the opposite, our method does not require any special initialization, which completely avoids to be trapped.
	\item For test issues, the GLE-driver Newton method \cite{fanWC2016smith} is only tested on an example with $r=2,\nnz(A_1)=n-1,\nnz(B_1)=2$; 
		the homotopy method \cite{zhangFCW2015homotopy} is only tested on the example \emph{Rail} with $r=2,\nnz(A_1)\approx 0.01\nnz(A),\nnz(B_1)\approx 0.1\nnz(B)$.
		In comparison, we test \emph{Rail} with as large as $r=5, \nnz(A_i)=\nnz(A),\nnz(B_i)=\nnz(B)$.
\end{itemize}


\section{Conclusion}\label{sec:concluding-remarks}
We have introduced an RADI-type method for computing the unique positive semi-definite stabilizing solution for large-scale stochastic continuous-time algebraic Riccati equations with sparse and low-rank matrices.
This efficient method benefits from the algebraic structure of the equations and builds upon the RADI method for solving classical continuous-time algebraic Riccati equations.
As evidence of its efficiency, this promising method successfully solves the standard benchmark problem, \emph{Rail}, involving four stochastic processes with the same sparsity at $n=79841,r=5, \nnz=553321\times 5$ in $100$ seconds with MATLAB on a standard PC without any prior information.

Similar to the RADI method, the performance of our approach heavily relies on the strategy for selecting shifts, among which we suggests employing the residual Hamiltonian shift in each iteration.
Due to lack of results on eigenvalue problem in the left semi-tensor product, we have to utilize the real part of a selected eigenvalue of the projected classical Hamiltonian matrix.

We believe that it is highly worthwhile to investigate the eigenvalue problem within the context of the left semi-tensor product, not only for its intrinsic significance but also for its potential to enhance the shift strategies and accelerate the method proposed in the paper.

\bibliographystyle{plain}
\bibliography{SCAREincorpReal-min}
\end{document}